\documentclass{siamltex}
\usepackage{graphicx} 
\usepackage{hyperref} 
\usepackage{amsmath,amssymb,verbatim} 
\usepackage{placeins} 
\usepackage{color}
\usepackage{cite}
\usepackage{subfigure}
\usepackage{hyperref}

\newcommand{\Ja}{J_{\alpha}}
\newcommand{\Phm}{\Phi_{\mu}}
\newcommand{\Phn}{\Phi_{\nu}}
\newcommand{\bbE}{\mathbb E}

\definecolor{darkred}{rgb}{.7,0,0}

\definecolor{darkgreen}{rgb}{0,0.5,0}

\definecolor{darkblue}{rgb}{0,0,0.7}

\definecolor{darkred}{rgb}{0.9,0.1,0.1}

\DeclareMathOperator{\spa}{span}

\DeclareMathOperator{\bigo}{O}

\DeclareMathOperator{\Erf}{Erf}
\DeclareMathOperator{\Erfc}{Erfc}

\newtheorem{algorithm}[theorem]{Algorithm}

\numberwithin{equation}{section}

\newcommand{\cA}{\mathcal A}

\newcommand{\cG}{\mathcal{G}}
\newcommand{\cH}{\mathcal{H}}

\newcommand{\E}{{\mathbb E}}
\newcommand{\bbP}{{\mathbb P}}

\newcommand{\R}{\mathbb{R}}

\newcommand{\eps}{\varepsilon}

\newcommand{\Dkl}{D_{{\rm KL}}}

\newcommand{\Dnm}{\Dkl(\nu \| \mu)}

\newcommand{\HS}{\mathcal{HS}(\cH)}

\newcommand{\Cni}{C_0^{-1}}

\newcommand{\Hg}{\cH_K}
\newcommand{\la}{\langle}
\newcommand{\ra}{\rangle}

\newcommand{\pg}{P}

\renewcommand{\epsilon}{\eps}

\title{Algorithms for Kullback-Leibler Approximation of Probability Measures in
Infinite Dimensions}
 
\author{F.~J.~Pinski\footnotemark[2] \and G.~Simpson\footnotemark[3]
  \and A.~M.~Stuart\footnotemark[4] \and H.~Weber\footnotemark[4]
\footnotetext[2]{Department of Physics, University of Cincinnati, 
  Cincinnati, OH 45221, USA}
\footnotetext[3]{Department of Mathematics, Drexel University,
  Philadelphia, PA 19104, USA}
\footnotetext[4]{Mathematics Institute, University of Warwick,
  Coventry CV4 7AL, UK}
}

\begin{document}

\maketitle

\begin{abstract} 
  In this paper we study algorithms to find a Gaussian approximation
  to a target measure defined on a Hilbert space of functions; the
  target measure itself is defined via its density with respect to a
  reference Gaussian measure. We employ the Kullback-Leibler
  divergence as a distance and find the best Gaussian
  approximation by minimizing this distance. It then follows that the
  approximate Gaussian must be equivalent to the Gaussian reference
  measure, defining a natural function space setting for the
  underlying calculus of variations problem. We introduce a
  computational algorithm which is well-adapted to the required
  minimization, seeking to find the mean as a function, and
  parameterizing the covariance in two different ways: through low
  rank perturbations of the reference covariance; and through
  Schr\"odinger potential perturbations of the inverse reference
  covariance. Two applications are shown: to a nonlinear inverse
  problem in elliptic PDEs, and to a conditioned diffusion process. We
  also show how the Gaussian approximations we obtain may be used to
  produce improved pCN-MCMC methods which are not only well-adapted to
  the high-dimensional setting, but also behave well with respect to
  small observational noise (resp. small temperatures) in the inverse
  problem (resp. conditioned diffusion).
\end{abstract}

\section{Introduction}
\label{sec:I}

Probability measures on infinite dimensional spaces arise in a variety
of applications, including the Bayesian approach to inverse problems
\cite{S10a} and conditioned diffusion processes
\cite{HSV11}. Obtaining quantitative information from such problems is
computationally intensive, requiring approximation of the infinite
dimensional space on which the measures live. We present a
computational approach applicable to this context: we demonstrate a
methodology for computing the best approximation to the measure, from
within a subclass of Gaussians. In addition we show how this best
Gaussian approximation may be used to speed-up Monte Carlo-Markov
chain (MCMC) sampling. The measure of ``best'' is taken to be the
Kullback-Leibler (KL) divergence, or relative entropy, a methodology
widely adopted in machine learning applications
\cite{bishop2006pattern}. In the recent paper~\cite{PSSW13},
KL-approximation by Gaussians was studied using the calculus of
variations. The theory from that paper provides the mathematical
underpinnings for the algorithms presented here.

\subsection{Abstract Framework}

Assume we are given a measure $\mu$ on the separable Hilbert space
$(\cH, \la \cdot , \cdot \ra, \, \| \cdot \|)$ equipped with the Borel
$\sigma$-algebra, specified by its density with respect to a reference
measure $\mu_0$. We wish to find the closest element $\nu$ to $\mu$,
with respect to KL divergence, from a subset $\cA$ of the Gaussian
probability measures on $\cH$. We assume the reference measure $\mu_0$
is itself a Gaussian $\mu_0=N(m_0,C_0)$ on $\cH$.  The measure $\mu$
is thus defined by
\begin{equation}
  \frac{d \mu}{ d \mu_0} (u)= \frac{1}{Z_\mu} \exp \big( - \Phm(u)\big), \label{e:target}
\end{equation}
where we assume that $\Phm:X \to \R$ is continuous on some Banach
space $X$ of full measure with respect to $\mu_0$, and that
$\exp(-\Phm(x))$ is integrable with respect to $\mu_0$. Furthermore,
$Z_\mu =\E^{\mu_0} \exp \big( - \Phm(u)\big)$ ensuring that $\mu$ is
indeed a {\em probability} measure.  We seek an approximation
$\nu=N(m,C)$ of $\mu$ which minimizes $\Dnm$, the KL divergence
between $\nu$ and $\mu$ in $\cA$. Under these assumptions it is
necessarily the case that $\nu$ is equivalent\footnote{Two measures
  are equivalent if they are mutually absolutely continuous.}  to
$\mu_0$ (we write $\nu \sim \mu_0$) since otherwise $\Dnm=\infty.$
This imposes restrictions on the pair $(m,C)$, and we build these
restrictions into our algorithms.  Broadly speaking, we will seek to
minimize over {\em all} sufficiently regular functions $m$, whilst we
will parameterize $C$ either through operators of finite rank, or
through a function appearing as a potential in an inverse covariance
representation.

Once we have found the best Gaussian approximation we will use this to
improve upon known MCMC methods.  Here, we adopt the perspective of
considering only MCMC methods that are well-defined in the
infinite-dimensional setting, so that they are robust to
finite-dimensional approximation \cite{CRSW13}. The best Gaussian
approximation is used to make Gaussian proposals within MCMC which are
simple to implement, yet which contain sufficient information about
$\Phm$ to yield significant reduction in the autocovariance of the
resulting Markov chain, when compared with the methods developed in
\cite{CRSW13}.

\subsection{Relation to Previous Work}

In addition to the machine learning applications mentioned above
\cite{bishop2006pattern}, approximation with respect to KL divergence
has been used in a variety of applications in the physical sciences,
including climate science \cite{Gershgorin:2012hu}, coarse graining
for molecular dynamics \cite{katsoulakis2013information,Shell:2008cj}
and data assimilation \cite{archambeau2007gaussian}.

On the other hand, improving the efficiency of MCMC algorithms is a
topic attracting a great deal of current interest, as many important
PDE based inverse problems result in target distributions $\mu$ for
which $\Phi_\mu$ is computationally expensive to evaluate.  In
\cite{Martin:2012fj}, the authors develop a stochastic Newton MCMC
algorithm, which resembles our improved pCN-MCMC Algorithm \ref{a2} in
that it uses Gaussian approximations that are adapted to the problem
within the proposal distributions.  However, while we seek to find
minimizers of KL in an offline computation, the work in
\cite{Martin:2012fj} makes a quadratic approximation of $\Phi_\mu$ at
each step along the MCMC sequence; in this sense it has similarities
with the Riemannian Manifold MCMC methods of
\cite{girolami2011riemann}.

As will become apparent, a serious question is how to characterize,
numerically, the covariance operator of the Gaussian measure $\nu$.
Recognizing that the covariance operator is compact, with decaying
spectrum, it may be well-approximated by a low rank matrix. Low rank
approximations are used in \cite{Martin:2012fj,Spantini:2014tb}, and
in the earlier work \cite{Flath:2011gm}.  In \cite{Flath:2011gm} the
authors discuss how, even in the case where $\mu$ is itself Gaussian,
there are significant computational challenges motivating the low rank
methodology.

Other active areas in MCMC methods for high dimensional problems
include the use of polynomial chaos expansions for proposals
\cite{Marzouk:2007vi}, and local interpolation of $\Phi_\mu$ to reduce
computational costs \cite{Conrad:2014vc}.  For methods which go beyond
MCMC, we mention the paper \cite{ElMoselhy:2012hn} in which the
authors present an algorithm for solving the optimal transport PDE
relating $\mu_0$ to $\mu$.

\subsection{Outline}

In Section \ref{s:scalar_example}, we examine these algorithms in the
context of a scalar problem, motivating many of our ideas.  The
general methodology is introduced in Section \ref{sec:G}, where we
describe the approximation of $\mu$ defined via \eqref{e:target} by a
Gaussian, summarizing the calculus of variations framework which
underpins our algorithms. We describe the problem of Gaussian
approximations in general, and then consider two specific
paramaterizations of the covariance which are useful in practice, the
first via finite rank perturbation of the covariance of the reference
measure $\mu_0$, and the second via a Schr\"odinger potential shift
from the inverse covariance of $\mu_0$. Section \ref{sec:RM} describes
the structure of the Euler-Lagrange equations for minimization, and
recalls the Robbins-Monro algorithm for locating the zeros of
functions defined via an expectation. In Section \ref{sec:MCMC} we
describe how the Gaussian approximation found via KL minimization can
be used as the basis for new MCMC methods, well-defined on function
space and hence robust to discretization, but also taking into account
the change of measure via the best Gaussian approximation.  Section
\ref{sec:N} contains illustrative numerical results, for a Bayesian
inverse problem arising in a model of groundwater flow, and in a
conditioned diffusion process, prototypical of problems in molecular
dynamics. We conclude in Section \ref{sec:C}.

\section{Scalar Example}
\label{s:scalar_example}

The main challenges and ideas of this work can be exemplified in a
scalar problem, which we examine here as motivation. Consider the measure $\mu^\eps$ defined via its density
with respect to Lebesgue measure:
\begin{equation}
  \label{e:scalar_dist}
  \mu^\eps(dx) = \frac{1}{Z_\eps} \exp\left(-\eps^{-1}V(x)\right)dx, \quad V:\R \to \R.
\end{equation}
$\eps>0$ is a small parameter.  Furthermore, let the potential $V$ be
such that $\mu^\eps$ is non-Gaussian.  As a concrete example, take
\begin{equation}
  \label{e:scalar_potential}
  V(x) =  x^4 + \tfrac{1}{2}x^2. 
\end{equation}
We now explain our ideas in the context of this simple example,
referring to algorithms which are detailed later; additional details
are given in Section \ref{s:scalar}.

In order to link to the infinite dimensional setting, where Lebesgue
measure is not defined and Gaussian measure is used as the reference
measure, we write $\mu^\eps$ via its density with respect to a unit
Gaussian $\mu_0 = N(0,1)$:
\[
\frac{d\mu^\eps}{d\mu_0} =\frac{\sqrt{2\pi}}{Z_\eps}
\exp\left(-{\eps}^{-1}V(x)+\tfrac{1}{2}x^2\right).
\]
We find the best fit $\nu = N(m, \sigma^2)$, optimizing $\Dnm$ over
$m\in \R$ and $\sigma>0$, noting that $\nu$ may be written as
\[
\frac{d\nu}{d\mu_0}=
\frac{\sqrt{2\pi}}{\sqrt{2\pi\sigma^2}}\exp\left(-\tfrac{1}{2\sigma^2}(x-m)^2
  +\tfrac{1}{2}x^2 \right).\] The change of measure is then
\begin{equation}
  \frac{d\mu^\eps}{d\nu} = \frac{\sqrt{2\pi \sigma^2}}{Z_\eps}\exp
  \left(-{\eps}^{-1}V(x) + \tfrac{1}{2\sigma^2 }(x-m)^2\right).
\end{equation}
For potential \eqref{e:scalar_potential}, $\Dkl$ can be integrated
analytically, yielding,
\begin{equation}
  \label{e:Dnm_scalar}
  \Dkl(\nu||\mu^\eps) = \tfrac{1}{2}\eps^{-1}\left(2m^4 + m^2 + 12 m^2 \sigma^2 + \sigma^2 +
    6\sigma^4 \right) -\tfrac{1}{2}+ \log{Z_\eps}-\log{\sqrt{2\pi \sigma^2}}.
\end{equation}
In subsection \ref{ss:mins} we illustrate an algorithm to find the
best Gaussian approximation numerically whilst subsection
\ref{ss:mcmcs} demonstrates how this minimizer maybe used to improve
MCMC methods.  Appendix \ref{s:SE} contains further
details of the numerical results, as well as a theoretical analysis of
the improved MCMC method for this problem.

\subsection{Estimation of the Minimizer}
\label{ss:mins}
The Euler-Lagrange equations for \eqref{e:Dnm_scalar} can then be
solved to obtain a minimizer $(m,\sigma)$ which satisfies $m=0$ and
\begin{equation}
  \label{e:sig_scalar}
  \sigma^2 = \tfrac{1}{24}\left(\sqrt{1+48\eps}-1\right)= \eps - 12\eps^2
  + \bigo(\eps^3).
\end{equation}
In more complex problems, $\Dnm$ is not analytically tractable and
only defined via expectation. In this setting, we rely on the
Robbins-Monro algorithm (Algorithm \ref{a:RMforKL}) to compute
solution of the Euler-Lagrange equations defining minimizers.  Figure
\ref{f:scalar_conv} depicts the convergence of the Robbins-Monro
solution towards the desired root at $\eps=0.01$, $(m,\sigma) \approx
(0, 0.0950)$ for our illustrative scalar example. It also shows that
$\Dnm$ is reduced.

\begin{figure}
  \begin{center}
    \subfigure[Convergence of
    $m_n$]{\includegraphics[width=6.25cm]{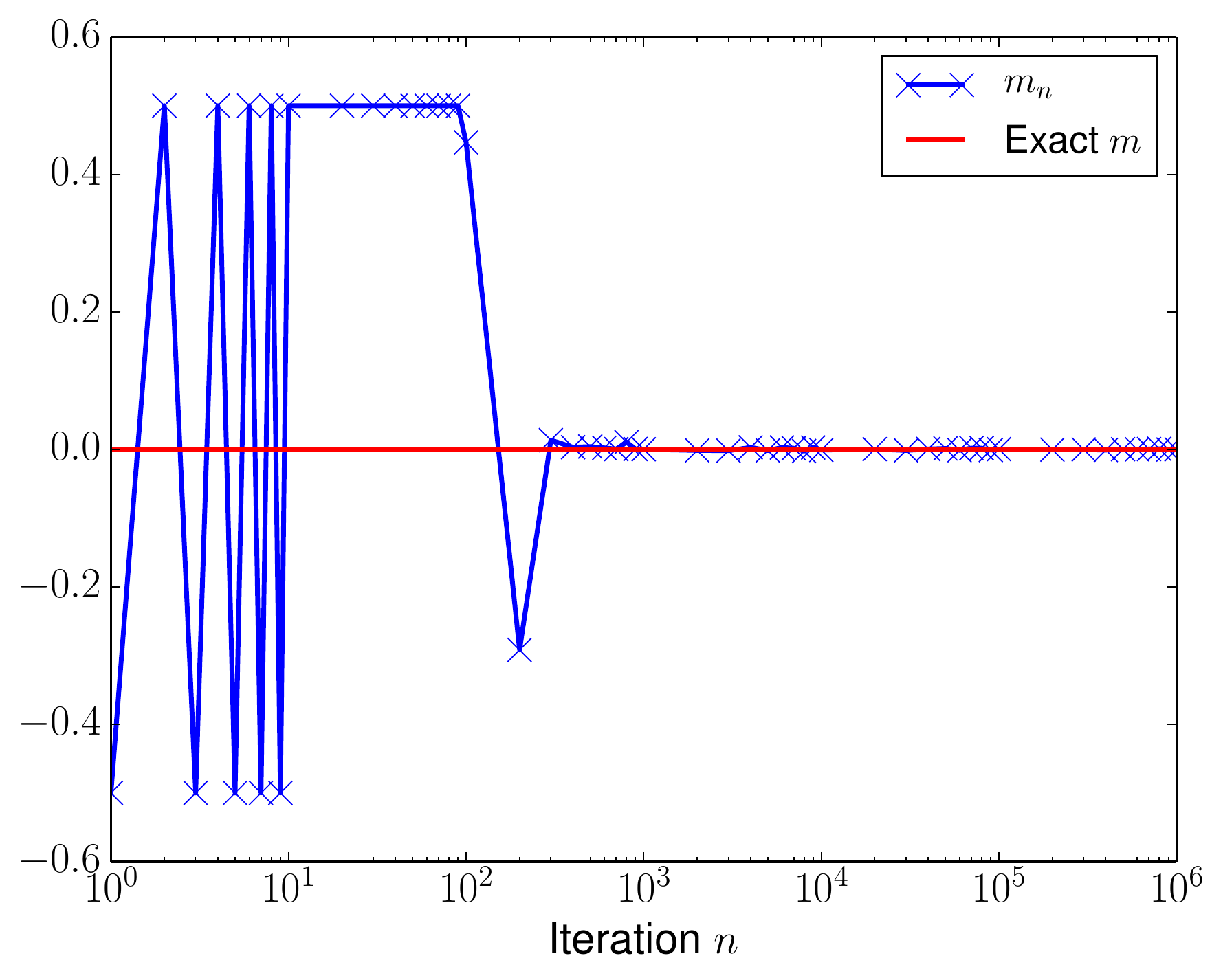}}
    \subfigure[Convergence of
    $\sigma_n$]{\includegraphics[width=6.25cm]{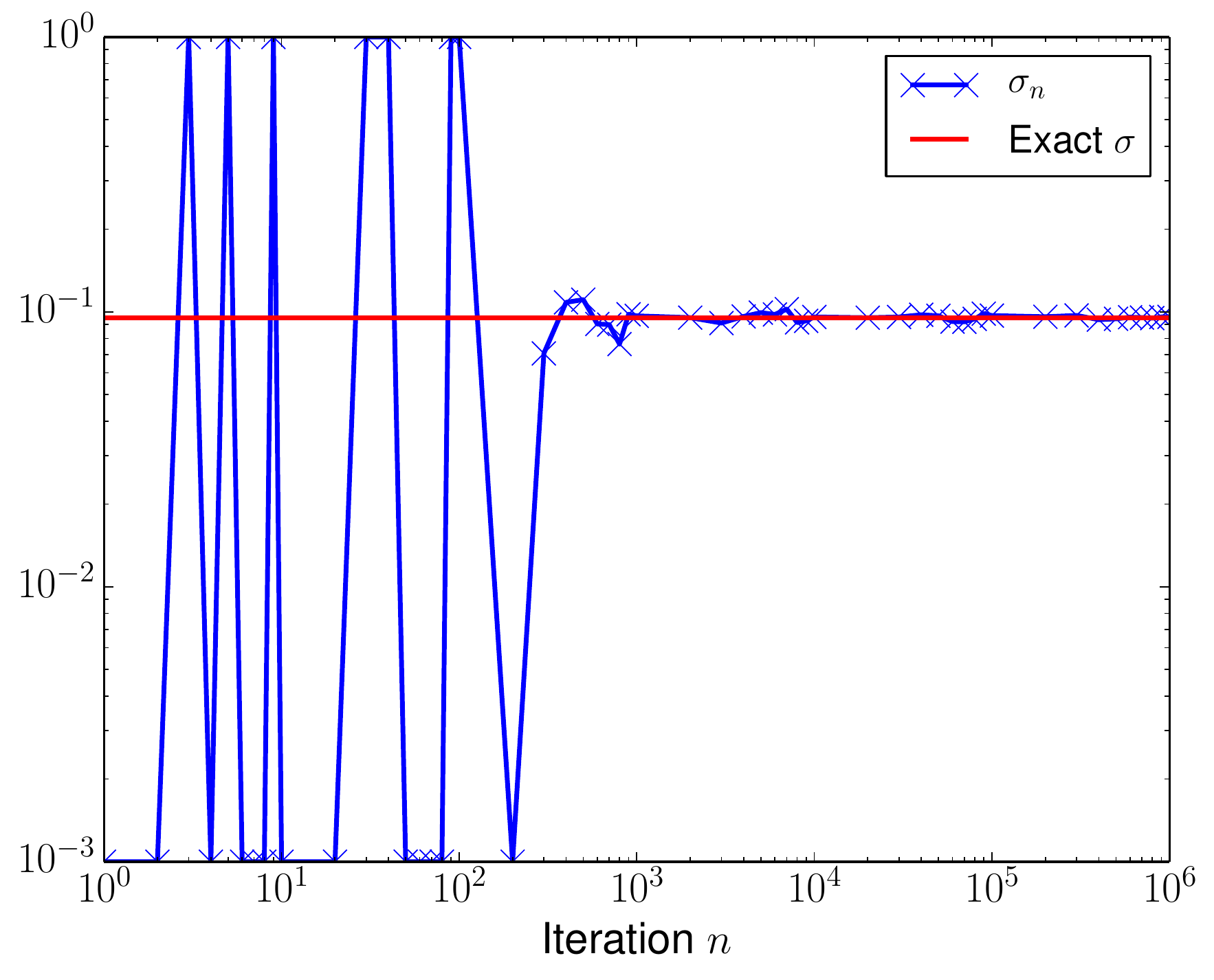}}

    \subfigure[Minimization of
    $\Dkl$]{\includegraphics[width=6.25cm]{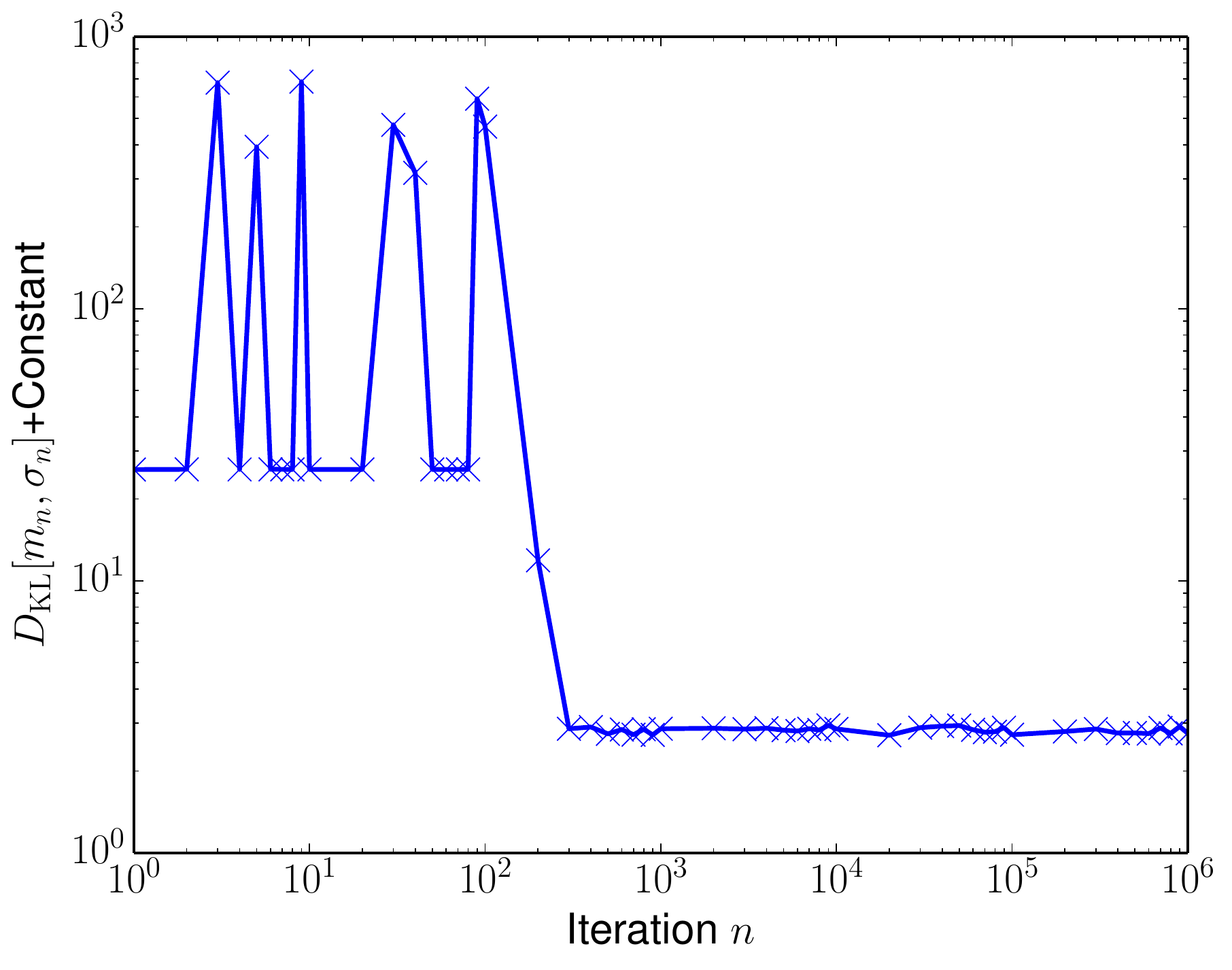}}
  \end{center}
  \caption{Convergence of $m_n$ and $\sigma_n$ towards the values
    found via deterministic root finding for the scalar problem with
    potential \eqref{e:scalar_potential} at $\eps =0.01$.  The
    iterates are generated using Algorithm \ref{a:RMforKL},
    Robbins-Monro applied to KL minimization.  Also plotted are values
    of KL divergence along the iteration sequence.  The true optimal
    value is recovered, and KL divergence is reduced.  To ensure
    convergence, $m_n$ is constrained to $[-.5,.5]$ and $\sigma_n$ is
    constrained to $[10^{-3}, 10^0]$.}
  \label{f:scalar_conv}
\end{figure}

\subsection{Sampling of the Target Distribution}
\label{ss:mcmcs}
Having obtained values of $m$ and $\sigma$ that minimize $\Dnm$, we
may use $\nu$ to develop an improved MCMC sampling algorithm for the
target measure $\mu^\eps$. We compare the performance of the standard
pCN method of Algorithm \ref{a1}, which uses no information about the
best Gaussian fit $\nu$, with the improved pCN Algorithm \ref{a2},
based on knowledge of $\nu.$ The improved performance, gauged by
acceptance rate and autocovariance, is shown in Figure
\ref{f:scalar_pCN_conv}.

\begin{figure}
  \begin{center}
    \subfigure[Acceptance
    Rate]{\includegraphics[width=6.25cm]{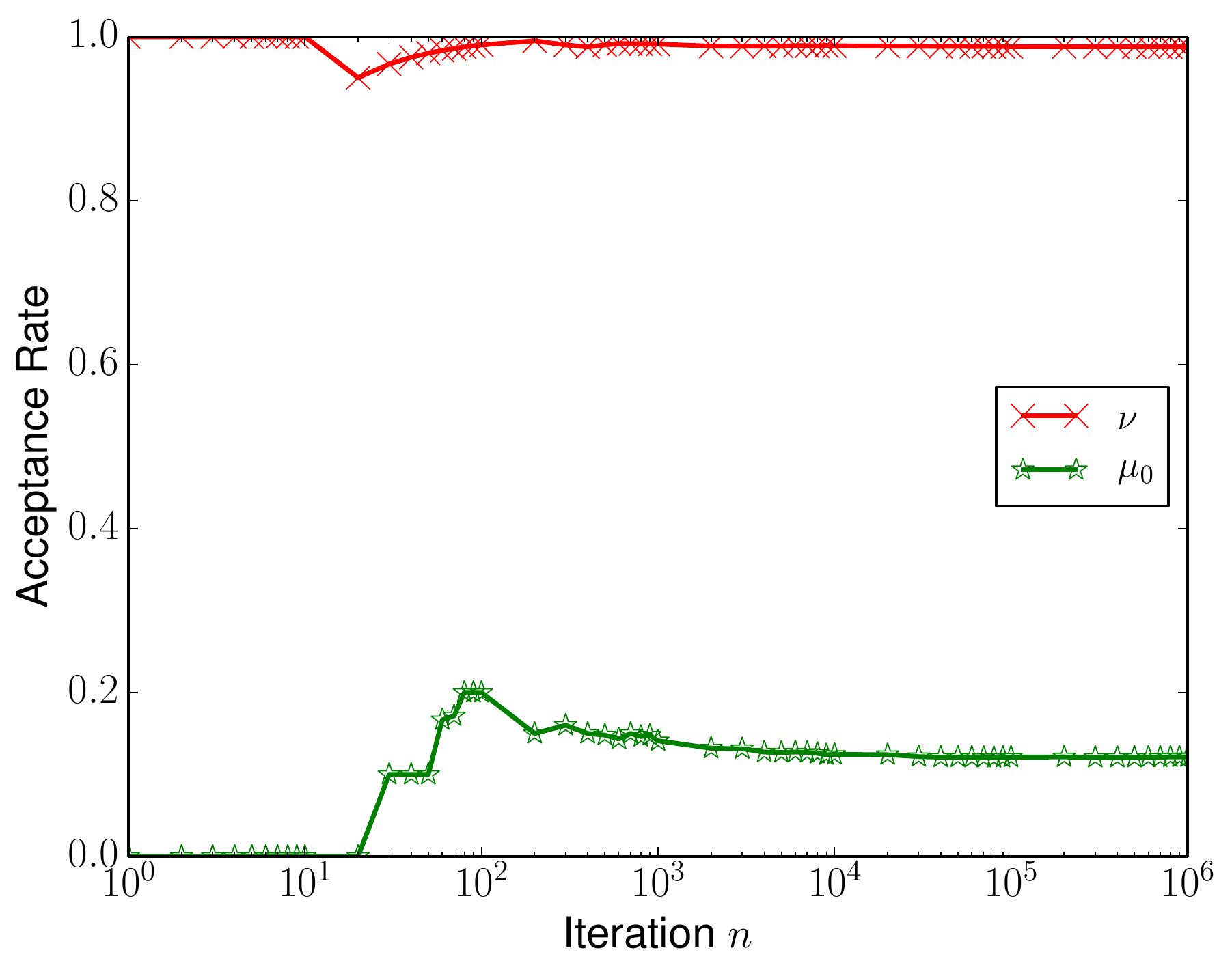}}
    \subfigure[Autocovariance]{\includegraphics[width=6.25cm]{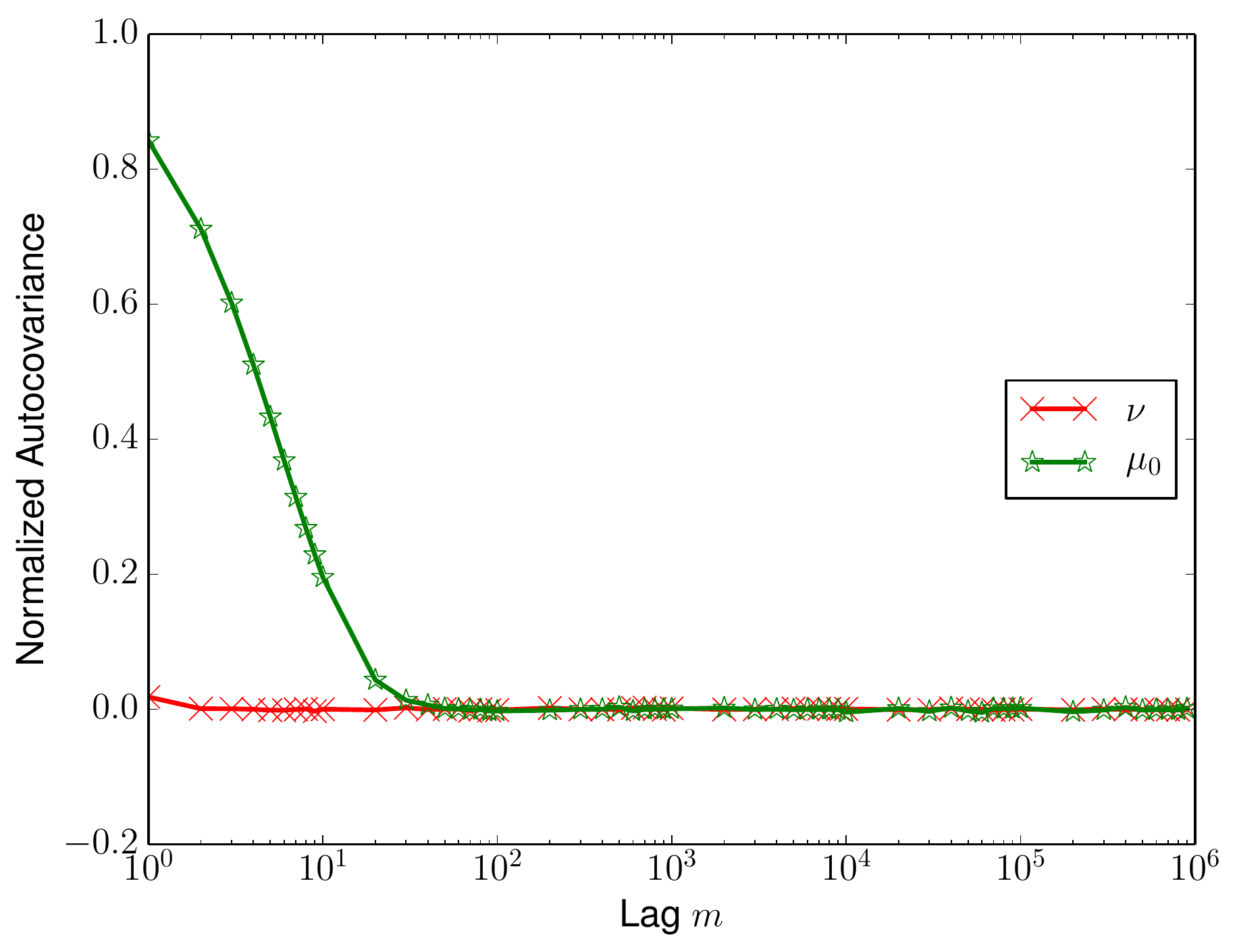}}
  \end{center}
  \caption{Acceptance rates and autocovariances for sampling from
    \eqref{e:scalar_dist} with potential \eqref{e:scalar_potential} at
    $\eps =0.01$.  The curves labeled $\nu$ correspond to the samples
    generated using our improved MCMC, Algorithm \ref{a2}, which uses
    the KL optimized $\nu$ for proposals.  The curves labeled $\mu_0$
    correspond to the samples generated using Algorithm \ref{a1},
    which relies on $\mu_0$ for proposals.  Algorithm \ref{a2} shows
    an order of magnitude improvement over Algorithm \ref{a1}. For
    clarity, only a subset of the data is plotted in the figures.}
  \label{f:scalar_pCN_conv}
\end{figure}

All of this is summarized by Figure \ref{f:scalar_summary}, which
shows the three distributions $\mu^\eps$, $\mu_0$ and KL optimized
$\nu$, together with a histogram generated by samples from the
KL-optimized MCMC Algorithm \ref{a2}.  Clearly, $\nu$ better
characterizes $\mu^\eps$ than $\mu_0$, and this is reflected in the
higher acceptance rate and reduced autocovariance.  Though this is
merely a scalar problem, these ideas are universal.  In all of our
examples, we have a non-Gaussian distribution we wish to sample from,
an uninformed reference measure which gives poor sampling performance,
and an optimized Gaussian which better captures the target measure and
can be used to improve sampling.

\begin{figure}
  \begin{center}
    \includegraphics[width=8cm]{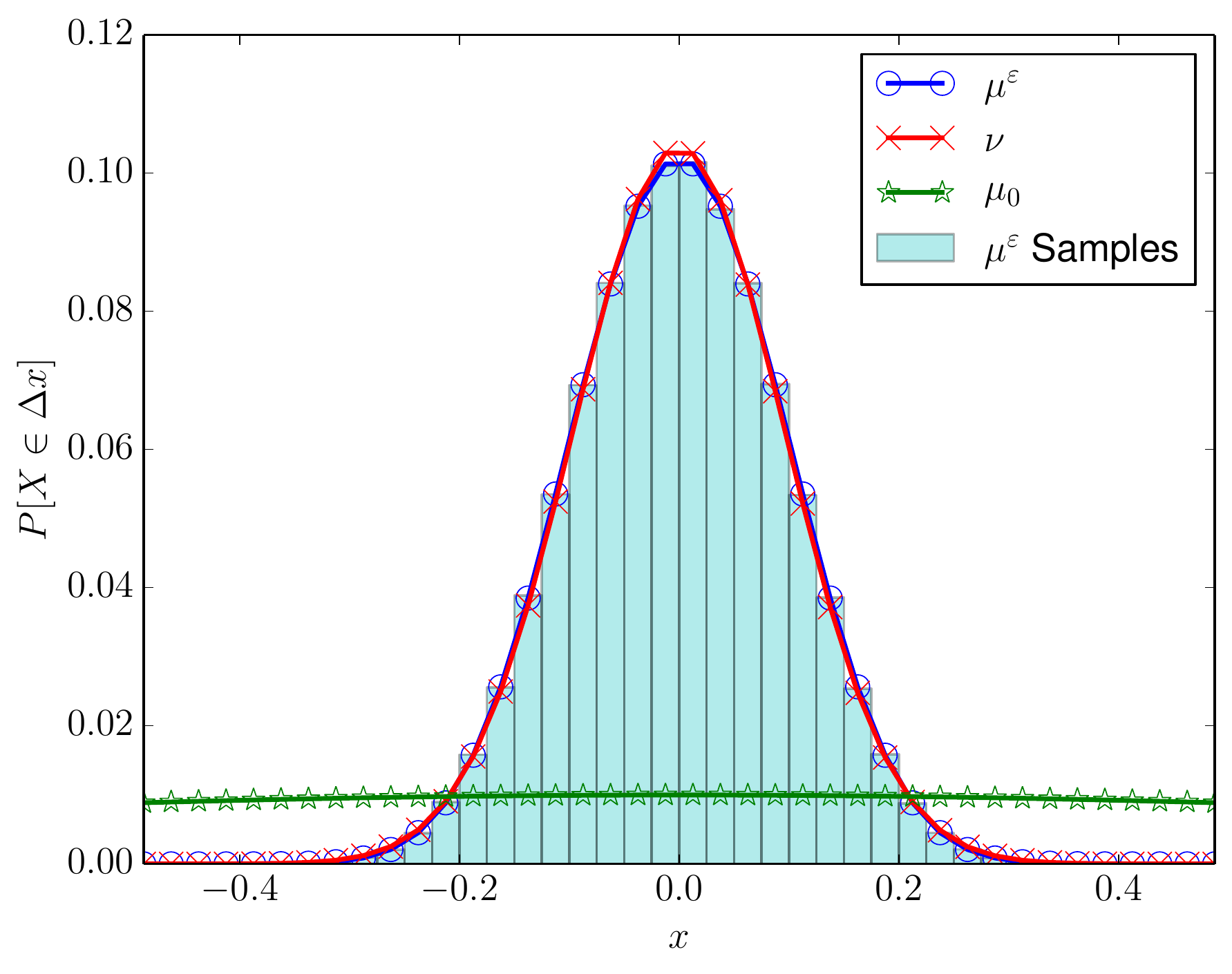}
  \end{center}
  \caption{Distributions of $\mu^\eps$ (target), $\mu_0$ (reference)
    and $\nu$ (KL-optimized Gaussian) for the scalar problem with
    potential \eqref{e:scalar_potential} at $\eps =0.01$.  Posterior
    samples have also been plotted, as a histogram.
    By inspection, $\nu$ better captures $\mu^\eps$, leading to
    improved performance. $\Delta x = 0.025$.}
  \label{f:scalar_summary}
\end{figure}

\section{Parameterized Gaussian Approximations}
\label{sec:G}

We start in subsection \ref{ssec:GS} by describing some general
features of the KL distance. Then in subsection \ref{ssec:GA} we
discuss the case where $\nu$ is Gaussian. Subsections \ref{ssec:FR}
and \ref{ssec:SP} describe two particular parameterizations of the
Gaussian class that we have found useful in practice.

\subsection{General Setting}
\label{ssec:GS}

Let $\nu$ be a measure defined by
\begin{equation}
  \frac{d \nu}{ d \mu_0} (u)= \frac{1}{Z_\nu} \exp \big( - \Phn(u)\big), \label{e:targeta}
\end{equation}
where we assume that $\Phn:X \to \R$ is continuous on $X.$ We aim to
choose the best approximation $\nu$ to $\mu$ given by \eqref{e:target}
from within some class of measures; this class will place restrictions
on the form of $\Phn.$ Our best approximation is found by choosing the
free parameters in $\nu$ to minimize the KL divergence between $\mu$
and $\nu$.  This is defined as
\begin{equation}
  \label{eq:KL}
  \Dkl(\nu\|\mu)=\int_{H}\log\Bigl(\frac{d\nu}{d\mu}(u)\Bigr)\nu(du)
  =\bbE^{\nu} \log\Bigl(\frac{d\nu}{d\mu}(u)\Bigr).
\end{equation}
Recall that $\Dkl(\cdot\|\cdot)$ is not symmetric in its two arguments
and our reason for choosing $\Dkl(\nu\|\mu)$ relates to the
possibility of capturing multiple modes individually; minimizing
$\Dkl(\mu\|\nu)$ corresponds to moment matching in the case where
$\cA$ is the set of all Gaussians \cite{bishop2006pattern,PSSW13}.

Provided $\mu_0 \sim\nu$, we can write
\begin{equation}
  \label{eq:mu2}
  \frac{d\mu}{d\nu}(u)=\frac{Z_{\nu}}{Z_{\mu}}\exp\bigl(-\Delta(u)\bigr),
\end{equation}
where
\begin{equation}
  \label{e:D}
  \Delta(u)=\Phm(u)-\Phn(u).
\end{equation}
Integrating this identity with respect to $\nu$ gives
\begin{equation}
  \label{eq:zm}
  \frac{Z_{\mu}}{Z_{\nu}}
  =\int_{H}\exp\bigl(-\Delta(u)\bigr)\nu(du)
  =\bbE^{\nu}\exp\bigl(-\Delta(u)\bigr).
\end{equation}
Combining \eqref{eq:KL} with \eqref{eq:mu2} and \eqref{eq:zm}, we have
\begin{equation}
  \label{eq:target}
  \Dkl(\nu\|\mu)=\bbE^{\nu} \Delta(u)+\log\Bigl(\bbE^{\nu}
  \exp\bigl(-\Delta(u)\bigr)\Bigr).
\end{equation}
The computational task in this paper is to minimize \eqref{eq:target}
over the parameters that characterize our class of approximating
measures $\cA$, which for us will be subsets of Gaussians.  These
parameters enter $\Phn$ and the normalization constant $Z_{\nu}.$ It
is noteworthy, however, that the normalization constants $Z_{\mu}$ and
$Z_{\nu}$ do not enter this expression for the distance and are hence
not explicitly needed in our algorithms.

To this end, it is useful to find the Euler-Lagrange equations of
\eqref{eq:target}.  Imagine that $\nu$ is parameterized by $\theta$
and that we wish to differentiate $J(\theta):=\Dkl(\nu\|\mu)$ with
respect to $\theta.$ We rewrite $J(\theta)$ as an integral with
respect to $\mu$, rather than $\nu$, differentiate under the integral,
and then convert back to integrals with respect to $\nu.$ From
\eqref{eq:mu2}, we obtain
\begin{equation}
  \label{eq:a}
  \frac{Z_{\nu}}{Z_{\mu}}=\bbE^{\mu} e^{\Delta}.
\end{equation}
Hence, from \eqref{eq:mu2},
\begin{equation}
  \label{eq:b}
  \frac{d\nu}{d\mu}(u)=\frac{e^{\Delta}}{\bbE^{\mu} e^{\Delta}}.
\end{equation}
Thus we obtain, from \eqref{eq:KL},
\begin{equation}
  \label{eq:KL2}
  J(\theta)=
  \bbE^{\mu} \Bigl(\frac{d\nu}{d\mu}(u)\log\Bigl(\frac{d\nu}{d\mu}(u)\Bigr)\Bigr)=\frac{\bbE^{\mu}\bigl(e^{\Delta}(\Delta-\log \bbE^{\mu} e^{\Delta})\bigr)}{\bbE^{\mu} e^{\Delta}},
\end{equation}
and
$$J(\theta)=\frac{\bbE^{\mu}\bigl(e^{\Delta}\Delta\bigr)}{\bbE^{\mu}\bigl(e^{\Delta}\bigr)}-\log \bbE^{\mu} e^{\Delta}.$$
Therefore, with $D$ denoting differentiation with respect to $\theta$,
$$DJ(\theta)=\frac{\bbE^{\mu}\bigl(e^{\Delta}\Delta\, D\Delta\bigr)}{\bbE^{\mu}\bigl(e^{\Delta}\bigr)}-\frac{\bbE^{\mu}\bigl(e^{\Delta}\Delta\bigr)\bbE^{\mu}\bigl(e^{\Delta}D\Delta\bigr)}{\Bigl(\bbE^{\mu}\bigl(e^{\Delta}\bigr)\Bigr)^2}.
$$
Using \eqref{eq:b} we may rewrite this as integration with respect to
$\nu$ and we obtain
\begin{equation}
  \label{e:DJ}
  DJ(\theta)=\bbE^{\nu}(\Delta\,
  D\Delta)-(\bbE^{\nu}\Delta)(\bbE^{\nu}D\Delta).
\end{equation}
Thus, this derivative is zero if and only if $\Delta$ and $D\Delta$
are uncorrelated under $\nu.$

\subsection{Gaussian Approximations}
\label{ssec:GA}

Recall that the reference measure $\mu_0$ is the Gaussian
$N(m_0,C_0).$ We assume that $C_0$ is a strictly positive-definite
trace class operator on $\cH$ \cite{Bog}. We let
$\{e_j,\lambda_j^2\}_{j=1}^\infty$ denote the eigenfunction/eigenvalue
pairs for $C_0$. Positive (resp. negative) fractional powers of $C_0$
are thus defined (resp. densely defined) on $\cH$ by the spectral
theorem and we may define $\cH^1:=D(C_0^{-\frac12})$, the
Cameron-Martin space of measure $\mu_0$. We assume that $m_0 \in
\cH^1$ so that $\mu_0$ is equivalent to $N(0,C_0)$, by the
Cameron-Martin Theorem \cite{Bog}.  We seek to approximate $\mu$ given
in \eqref{e:target} by $\nu \in \cA$, where $\cA$ is a subset of the
Gaussian measures on $\cH$.  It is shown in \cite{PSSW13} that this
implies that $\nu$ is equivalent to $\mu_0$ in the sense of measures
and this in turn implies that $\nu=N(m,C)$ where $m\in E$ and
\begin{equation}\label{e:covariance}
  \Gamma:= C^{-1} - \Cni 
\end{equation}
satisfies
\begin{equation}\label{e:FHH}
  \big\| C_0^{\frac12} \Gamma C_0^{\frac12} \big\|_{\HS}^2     < \infty;
\end{equation}
here $\HS$ denotes the space of Hilbert-Schmidt operators on $\cH$.

For practical reasons, we do not attempt to recover $\Gamma$ itself,
but instead introduce low dimensional parameterizations.  Two such
parameterizations are introduced in this paper.  In one, we introduce
a finite rank operator, associated with a vector $\phi \in \R^n$.  In
the other, we employ a multiplication operator characterized by a
potential function $b$.  In both cases, the mean $m$ is an element of
$\cH^1$.  Thus minimization will be over either $(m,\phi)$ or $(m,b)$.

In this Gaussian case the expressions for $\Dkl$ and its derivative,
given by equations \eqref{eq:target} and \eqref{e:DJ}, can be
simplified.  Defining
\begin{equation}
  \label{e:Phinu}
  \Phi_\nu(u) = {- \langle u-m,
    m-m_0\rangle_{C_0} + \tfrac{1}{2}\langle u-m, \Gamma(u-m)\rangle -
    \tfrac{1}{2}\|m-m_0\|_{C_0}^2},
\end{equation}
we observe that, assuming $\nu \sim \mu_0$,
\begin{equation}
  \begin{split}
    \frac{d\nu}{d\mu_0}
    & \propto \exp \bigl(- \Phi_\nu(u) \bigr).
  \end{split}
\end{equation}
This may be substituted into the definition of $\Delta$ in
\eqref{e:D}, and used to calculate $J$ and $DJ$ according to
\eqref{eq:KL2} and \eqref{e:DJ}. However, we may derive alternate
expressions as follows. Let $\rho_0 = N(0, C_0)$, the centered version
of $\mu_0$, and $\nu_0=M(0,C)$ the centered version of $\nu.$ Then,
using the Cameron-Martin formula,
\begin{equation}
  \label{e:Znu0}
  Z_\nu = \bbE^{\mu_0}\exp({-\Phi_\nu}) = \bbE^{\rho_0}\exp({-\Phi_{\nu_0}})  =\Bigl( \bbE^{\nu_0}\exp({\Phi_{\nu_0}})\Bigr)^{-1}=Z_{\nu_0},
\end{equation}
where
\begin{equation}
  \label{e:Phi0}
  \Phi_{\nu_0} =\tfrac{1}{2}\langle u, \Gamma u\rangle.
\end{equation}
We also define a reduced $\Delta$ function which will play a role in
our computations:
\begin{equation}
  \label{e:D0}
  \Delta_0(u) \equiv{\Phi_\mu(u+m) - \tfrac{1}{2}\langle u,
    \Gamma
    u\rangle}.
\end{equation}
The consequence of these calculations is that, in the Gaussian case,
\eqref{eq:target} is
\begin{equation}
  \label{e:dkl_gauss}
  \begin{split}
    \Dkl(\nu||\mu) &= \bbE^{\nu}\Delta - \log Z_{\nu_0} + \log Z_\mu\\
    & =\bbE^{\nu_0}[\Delta_0] + \tfrac{1}{2}\| m - m_0\|_{C_0}^2 +
    \log \bbE^{\nu_0}\exp(\tfrac{1}{2}\langle u, \Gamma u\rangle)+
    \log Z_\mu.
  \end{split}
\end{equation}
Although the normalization constant $Z_\mu$ now enters the expression
for the objective function, it is irrelevant in the minimization since
it does not depend on the unknown parameters in $\nu$.  To better see
the connection between \eqref{eq:target} and \eqref{e:dkl_gauss}, note
that
\begin{equation}
  \frac{Z_\mu}{Z_{\nu_0}}=\frac{Z_\mu}{Z_{\nu}} =
  \frac{\E^{\mu_0}\exp(-\Phi_\mu)}{\E^{\mu_0}\exp(-\Phi_\nu)} = \E^{\nu}\exp(-\Delta).
\end{equation}

Working with \eqref{e:dkl_gauss}, the Euler-Lagrange equations to be
solved are:
\begin{subequations}
  \label{e:DJ_Gauss}
  \begin{align}
    \label{e:DJ_Gauss_m}
    D_mJ(m,\theta) &= \bbE^{\nu_0}D_u\Phi_\mu(u+m)  + C_0^{-1} (m-m_0),\\
    \label{e:DJ_Gauss_C}
    D_\theta J(m,\theta) & = \bbE^{\nu_0}(\Delta_0\,
    D_\theta\Delta_0)-(\bbE^{\nu_0}\Delta_0)(\bbE^{\nu_0}D_\theta\Delta_0).
  \end{align}
\end{subequations}
Here, $\theta$ is any of the parameters that define the covariance
operator $C$ of the Gaussian $\nu$.  Equation \eqref{e:DJ_Gauss_m} is
obtained by direct differentiation of \eqref{e:dkl_gauss}, while
\eqref{e:DJ_Gauss_C} is obtained in the same way as \eqref{e:DJ}.
These expressions are simpler for computations for two reasons.
First, for the variation in the mean, we do not need the full
covariance expression of \eqref{e:DJ}.  Second, $\Delta_0$ has fewer
terms to compute.

\subsection{Finite Rank Parameterization}
\label{ssec:FR}

Let $\pg$ denote orthogonal projection onto $\Hg := \spa\{e_1, \ldots
, e_K\}$ the span of the first $K$ eigenvectors of $C_0$ and define
$Q=I-P.$ We then parameterize the covariance $C$ of $\nu$ in the form
\begin{equation}
  \label{e:lowrank}
  C^{-1}=\bigl(QC_0Q\bigr)^{-1}+\chi, \quad \chi=\sum_{i,j \le K} \gamma_{ij} e_i \otimes e_j.
\end{equation}
In words $C^{-1}$ is given by the inverse covariance $C_0^{-1}$ of
$\mu_0$ on $Q\cH$, and is given by $\chi$ on $P\cH.$ Because $\chi$ is
necessarily symmetric it is essentially parametrized by a vector
$\phi$ of dimension $n=\frac12 K(K+1).$ We minimize $J(m,\phi):=
\Dkl(\nu\|\mu)$ over $(m,\phi) \in \cH^1 \times \R^n.$ This is a
well-defined minimization problem as demonstrated in Example 3.7 of
\cite{PSSW13} in the sense that minimizing sequences have weakly
convergent subsequences in the admissible set.  Minimizers need not be
unique, and we should not expect them to be, as multimodality is to be
expected, in general, for measures $\mu$ defined by \eqref{e:target}.

\subsection{Schr\"odinger Parameterization}
\label{ssec:SP}

In this subsection we assume that $\cH$ comprises a Hilbert space of
functions defined on a bounded open subset of $\R^d$.  We then seek
$\Gamma$ given by \eqref{e:covariance} in the form of a multiplication
operator so that $(\Gamma u)(x)=b(x)u(x).$ Whilst minimization over
the pair $(m,\Gamma)$, with $m \in \cH^1$ and $\Gamma$ in the space of
linear operators satisfying \eqref{e:FHH}, is well-posed
\cite{PSSW13}, minimizing sequences $\{m_k,\Gamma_k\}_{k \ge 1}$ with
$(\Gamma_k u)(x)=b_k(x)u(x)$ can behave very poorly with respect to
the sequence $\{b_k\}_{k \ge 1}.$ For this reason we regularize the
minimization problem and seek to minimize
$$\Ja(m,b)=J(m,b)+\tfrac{\alpha}{2}\|b\|_r^2$$
where $J(m,b):= \Dkl(\nu\|\mu)$ and $\|\cdot\|_r$ denotes the Sobolev
space $H^r$ of functions on $\R^d$ with $r$ square integrable
derivatives, with boundary conditions chosen appropriately for the
problem at hand.  The minimization of $J_\alpha(m,b)$ over $(m,b) \in
\cH \times H^r$ is well-defined; see Section 3.3 of \cite{PSSW13}.

\section{Robbins-Monro Algorithm}
\label{sec:RM}

In order to minimize $\Dnm$ we will use the Robbins-Monro algorithm
\cite{Robbins:1950ua,Asmussen:2010aa,Pasupathy:2011cs,Kushner:2003aa}.
In its most general form this algorithm calculates zeros of functions
defined via an expectation. We apply it to the Euler-Lagrange
equations to find critical points of a non-negative objective
function, defined via an expectation. This leads to a form of gradient
descent in which we seek to integrate the equations
\begin{equation*}
  \dot m = -D_m\Dkl, \quad  \dot \theta = - D_\theta \Dkl
\end{equation*}
until they have reached a critical point. This requires two
approximations.  First, as \eqref{e:DJ_Gauss} involve expectations,
the right hand sides of these differential equations are evaluated
only approximately, by sampling.  Second, a time discretization must
be introduced.  The key idea underlying the algorithm is that,
provided the step-length of the algorithm is sent to zero judiciously,
the sampling error averages out and is diminished as the step length
goes to zero.

\subsection{Background on Robbins-Monro}
\label{s:RMback}
In this section we review some of the structure in the Euler-Lagrange
equations for the desired minimization of $\Dnm$. We then describe the
particular variant of the Robbins-Monro algorithm that we use in
practice. Suppose we have a parameterized distribution, $\nu_\theta$,
from which we can generate samples, and we seek a value $\theta$ for
which
\begin{equation}
  \label{e:RM_generic}
  f(\theta) \equiv \E^{\nu_\theta}[Y] = 0, \quad Y \sim \nu_\theta.
\end{equation}
Then an estimate of the zero, $\theta_\star$, can be obtained via the
recursion
\begin{equation}
  \label{e:RM1}
  \theta_{n+1} = \theta_n - a_n \sum_{m=1}^M \tfrac{1}{M} Y_m^{(n)}, \quad
  Y_m^{(n)} \sim \nu_{\theta_n},\quad \text{i.i.d.}
\end{equation}
Note that the two approximations alluded to above are included in this
procedure: sampling and (Euler) time-discretization.  The methodology
may be adapted to seek solutions to
\begin{equation}
  f(\theta) \equiv \E^{\nu}[F(Y;\theta)]=0, \quad Y\sim \nu,
\end{equation}
where $\nu$ is a given, fixed, distribution independent of the
parameter $\theta$. (This setup arises, for example, in
\eqref{e:DJ_Gauss_m}, where $\nu_0$ is fixed and the parameter in
question is $m.$) Letting $Z = F(Y;\theta)$, this induces a
distribution $\eta_\theta(dz) = \nu(F^{-1}(dz;\theta))$, where the
pre-image is with respect to the $Y$ argument.  Then $f(\theta) =
\E^{\eta_\theta}[Z]$ with $Z \sim \eta_\theta$, and this now has the
form of \eqref{e:RM_generic}.  As suggested in the extensive
Robbins-Monro literature, we take the step sequence to satisfy
\begin{equation}
  \label{e:seq_cond}
  \sum_{n=1}^\infty a_n = \infty, \quad \sum_{n=1}^\infty a_n^2 < \infty.
\end{equation}
A suitable choice of $\{a_n\}$ is thus $a_n = a_0 n^{-\gamma}$,
$\gamma \in (1/2,1]$.  The smaller the value of $\gamma$, the more
``large'' steps will be taken, helping the algorithm to explore the
configuration space.  On the other hand, once the sequence is near the
root, the smaller $\gamma$ is, the larger the Markov chain variance
will be.  In addition to the choice of the sequence $a_n$,
\eqref{e:RM_generic} introduces an additional parameter, $M$, the
number of samples to be generated per iteration.  See
\cite{Byrd:2012bd,Asmussen:2010aa} and references therein for
commentary on sample size.

The conditions needed to ensure convergence, and what kind of
convergence, have been relaxed significantly through the years.  In
their original paper, Robbins and Monro assumed that $Y\sim
\mu_\theta$ were almost surely uniformly bounded, with a constant
independent of $\theta$.  If they also assumed that $f(\theta)$ was
monotonic and $f'(\theta_\star) >0$, they could obtain convergence in
$L^2$.  With somewhat weaker assumptions, but still requiring that the
zero be simple, Blum developed convergence with probability one,
\cite{Blum:1954tf} .  All of this was subsequently generalized to the
arbitrary finite dimensional case; see
\cite{Asmussen:2010aa,Kushner:2003aa,Pasupathy:2011cs}.

As will be relevant to this work, there is the question of the
applicability to the infinite dimensional case when we seek, for
instance, a mean function in a separable Hilbert space.  This has also
been investigated; see \cite{Yin:1990wv, Dvoretzky:1986ca} along with
references mentioned in the preface of \cite{Kushner:2003aa}.  In this
work, we do not verify that our problems satisfy convergence criteria;
this is a topic for future investigation.

A variation on the algorithm that is commonly applied is the
enforcement of constraints which ensure $\{\theta_n\}$ remain in some
bounded set; see \cite{Kushner:2003aa} for an extensive discussion.
We replace \eqref{e:RM1} by
\begin{equation}
  \label{e:RMcon}
  \theta_{n+1} = \Pi_D\left[\theta_n - a_n \sum_{m=1}^M \tfrac{1}{M} Y_m^{(n)}\right], \quad
  Y_m^{(n)} \sim \nu_{\theta_n},\quad \text{i.i.d.},
\end{equation}
where $D$ is a bounded set, and $\Pi_D(x)$ computes the point in $D$
nearest to $x$.  This is important in our work, as the parameters that
define must correspond to covariance operators. They must be positive
definite, symmetric, and trace-class. Our method automatically
produces symmetric trace-class operators, but the positivity has to be
enforced by a projection.

\subsection{Robbins-Monro Applied to KL}

We seek minimizers of $\Dkl$ as stationary points of the associated
Euler-Lagrange equations, \eqref{e:DJ_Gauss}.  Before applying
Robbins-Monro to this problem, we observe that we are free to
precondition the Euler-Lagrange equations.  In particular, we can
apply bounded, positive, invertible operators so that pre-conditioned
gradient will lie in the same function space as the parameter; this
makes the iteration scheme well posed.  For \eqref{e:DJ_Gauss_m}, we
have found pre-multiplying by $C_0$ to be sufficient.  For
\eqref{e:DJ_Gauss_C}, the operator will be problem specific, depending
on how $\theta$ parameterizes $C$, and also if there is a
regularization.  We denote the preconditioner for the second equation
by $B_\theta$.  Thus, the preconditioned Euler-Lagrange equations are
\begin{subequations}
  \label{e:preDJ_Gauss}
  \begin{align}
    \label{e:preDJ_Gauss_m}
    0= &C_0\bbE^{\nu_0}D_u\Phi_\mu(u+m)  +  (m-m_0),\\
    \label{e:preDJ_Gauss_C}
    0= &B_\theta\left[\bbE^{\nu_0}(\Delta_0\,
      D_\theta\Delta_0)-(\bbE^{\nu_0}\Delta_0)(\bbE^{\nu_0}D_\theta\Delta_0)\right].
  \end{align}
\end{subequations}
We must also ensure that $m$ and $\theta$ correspond to a well defined
Gaussian; $C$ must be a covariance operator.  Consequently, the
Robbins-Monro iteration scheme is:
\begin{algorithm}
  \label{a:RMforKL}
  \begin{enumerate}
  \item Set $n=0$.  Pick $m_0$ and $\theta_0$ in the admissible set,
    and choose a sequence $\{a_n\}$ satisfying \eqref{e:seq_cond}
  \item Update $m_n$ and $\theta_n$ according to:
    \begin{subequations}
      \label{e:RM_preKL}
      \begin{align}
        m_{n+1} &= \Pi_m\left[m_n - a_n
          \left\{C_0\left(\sum_{\ell=1}^M \tfrac{1}{M}\cdot
              D_u\Phi_\mu(u_\ell)\right) + m_n - m_0\right\}\right],\\
        \begin{split}
          \theta_{n+1} &= \Pi_\theta\left[\theta_n - a_n B_\theta
            \left\{\sum_{\ell=1}^M \tfrac{1}{M}\cdot
              \Delta_0(u_\ell) D_\theta \Delta_0(u_\ell)\right.\right.\\
          &\hspace{2cm}\left. \left.-\left(\sum_{\ell=1}^M
                \tfrac{1}{M}\cdot \Delta_0(u_\ell)
              \right)\left(\sum_{\ell=1}^M \tfrac{1}{M}\cdot D_\theta
                \Delta_0(u_\ell)\right)\right\}\right].
        \end{split}
      \end{align}
    \end{subequations}
  \item $n\to n+1$ and return to 2
  \end{enumerate}
\end{algorithm}
Typically, we have some {\it a priori} knowledge of the magnitude of
the mean.  For instance, $m\in H^1([0,1];\R^1)$ may correspond to a
mean path, joining two fixed endpoints, and we know it to be
confined to some interval $[\underline{m},\overline{m}]$.  In this
case we choose
\begin{equation}
  \Pi_m(f)(t) = \min\{\max\{f(t),\underline{m}\},\overline{m}\}, \quad 0<t<1.
\end{equation}
For $\Pi_\theta$, it is necessary to compute part of the spectrum of
the operator that $\theta$ induces, check that it is positive, and if
it is not, project the value to something satisfactory.  In the case
of the finite rank operators discussed in Section \ref{ssec:FR}, the
matrix $\boldsymbol{\gamma}$ must be positive.  One way of handing
this, for symmetric real matrices is to make the following choice:
\begin{equation}
  \Pi_\theta(A) = X \diag\{\min\{  \max\{\boldsymbol{\lambda}, \underline
  \lambda\}, \overline{\lambda}\}\}X^T,
\end{equation}
where $A = X \diag\{\boldsymbol{\lambda}\}X^{T}$ is the spectral
decomposition, and $\underline{\lambda}$ and $\overline{\lambda}$ are
constants chosen {\it a priori}.  It can be shown that this projection
gives the closest, with respect to the Frobenius norm, symmetric
matrix with spectrum constrained to $[\underline{\lambda},
\overline{\lambda}]$, \cite{Higham:1988jg}.\footnote{Recall that the
  Frobenius norm is the finite dimensional analog of the
  Hilbert-Schmidt norm.}

\section{Improved MCMC Sampling}
\label{sec:MCMC}

The idea of the Metropolis-Hastings variant of MCMC is to create an
ergodic Markov chain which is reversible, in the sense of Markov
processes, with respect to the measure of interest; in particular the
measure of interest is invariant under the Markov chain.  In our case
we are interested in the measure $\mu$ given by \eqref{e:target}.
Since this measure is defined on an infinite dimensional space it is
advisable to use MCMC methods which are well-defined in the infinite
dimensional setting, thereby ensuring that the resulting methods have
mixing rates independent of the dimension of the finite dimensional
approximation space. This philosophy is explained in the paper
\cite{CRSW13}. The pCN algorithm is perhaps the simplest MCMC method
for \eqref{e:target} meeting these requirements. It has the following
form:

\begin{algorithm}

  Define $a_{\mu}(u,v):=\min\{1,\exp\bigl(\Phm(u)-\Phm(v)\bigr)\}.$

  \begin{enumerate}

  \item Set $k=0$ and Pick $u^{(0)}$

  \item $v^{(k)}=m_0+\sqrt{(1-\beta^2)}(u^{(k)}-m_0)+\beta \xi^{(k)},
    \quad \xi^{(k)} \sim N(0,C_0)$

  \item Set $u^{(k+1)}=v^{(k)}$ with probability
    $a_{\mu}(u^{(k)},v^{(k)})$

  \item Set $u^{(k+1)}=u^{(k)}$ otherwise

  \item $k \to k+1$ and return to 2

  \end{enumerate}
  \label{a1}
\end{algorithm}

This algorithm has a spectral gap which is independent of the
dimension of the discretization space under quite general assumptions
on $\Phi_\mu$ \cite{hairer2011spectral}.  However, it can still behave
poorly if $\Phi_\mu$, or its gradients, are large. This leads to poor
acceptance probabilities unless $\beta$ is chosen very small so that
proposed moves are localized; either way, the correlation decay is
slow and mixing is poor in such situations.  This problem arises
because the underlying Gaussian $\mu_0$ used in the algorithm
construction is far from the target measure $\mu$. This suggests a
potential resolution in cases where we have a good Gaussian
approximation to $\mu$, such as the measure $\nu$.  Rather than basing
the pCN approximation on \eqref{e:target} we base it on
\eqref{eq:mu2}; this leads to the following algorithm:

\begin{algorithm}

  Define
  $a_{\nu}(u,v):=\min\{1,\exp\bigl(\Delta(u)-\Delta(v)\bigr)\}.$

  \begin{enumerate}

  \item Set $k=0$ and Pick $u^{(0)}$

  \item $v^{(k)}=m+\sqrt{(1-\beta^2)}(u^{(k)}-m)+\beta \xi^{(k)},
    \quad \xi^{(k)} \sim N(0,C)$

  \item Set $u^{(k+1)}=v^{(k)}$ with probability
    $a_{\nu}(u^{(k)},v^{(k)})$

  \item Set $u^{(k+1)}=u^{(k)}$ otherwise

  \item $k \to k+1$ and return to 2

  \end{enumerate}
  \label{a2}
\end{algorithm}

We expect $\Delta$ to be smaller than $\Phi$, at least in regions of
high $\mu$ probability. This suggests that, for given $\beta$,
Algorithm \ref{a2} will have better acceptance probability than
Algorithm \ref{a1}, leading to more rapid sampling. We show in what
follows that this is indeed the case.

\section{Numerical Results}
\label{sec:N}

In this section we describe our numerical results. These concern both
solution of the relevant minimization problem, to find the best
Gaussian approximation from within a given class using Algorithm
\ref{a:RMforKL} applied to the two parameterizations given in
subsections \ref{ssec:FR} and \ref{ssec:SP}, together with results
illustrating the new pCN Algorithm \ref{a2} which employs the best
Gaussian approximation within MCMC. We consider two model problems: a
Bayesian Inverse problem arising in PDEs, and a Conditioned Diffusion
problem motivated by molecular dynamics.  Some details on the path
generation algorithms used in these two problems are given in Appendix
\ref{a:samples}.

\subsection{Bayesian Inverse Problem}
\label{ssec:BIP}

We consider an inverse problem arising in groundwater flow. The
forward problem is modelled by the Darcy constitutive model for porous
medium flow. The objective is to find $p \in V:=H^1$ given by the
equation
\begin{subequations}
  \label{e:perm}
  \begin{align}
    -\nabla\cdot \bigl(\exp(u)\nabla p\bigr)&=0, \quad x\in D,\\
    p&=g, \quad x\in \partial D.
  \end{align}
\end{subequations}
The inverse problem is to find $u \in X=L^\infty(D)$ given noisy
observations
$$y_j=\ell_j(p)+\eta_j,$$
where $\ell_j \in V^*$, the space of continuous linear functionals on
$V$.  This corresponds to determining the log permeability from
measurements of the hydraulic head (height of the
water-table). Letting $\cG(u) = {\ell}(p(\cdot;u))$, the solution
operator of \eqref{e:perm} composed with the vector of linear
functionals ${\ell} = (\ell_j)^T$.  We then write, in vector form,
$$y=\cG(u)+\eta.$$
We assume that $\eta \sim N(0,\Sigma)$ and place a Gaussian prior
$N(m_0,C_0)$ on $u$. Then the Bayesian inverse problem has the form
\eqref{e:target} where
$$\Phi(u):=\frac12\bigl\|\Sigma^{-\frac12}\bigl(y-\cG(u)\bigr)\bigr\|^2.$$

We consider this problem in dimension one, with $\Sigma = \gamma^2 I$,
and employing pointwise observation at points $x_j$ as the linear
functionals $\ell_j$.  As prior we take the Gaussian $\mu_0 = N(0,
C_0)$, with
\begin{equation*}
  C_0 = \delta \left(-\frac{d^2}{dx^2} \right)^{-1},
\end{equation*}
restricted to the subspace of $L^2(0,1)$ of periodic mean zero
functions.  In one dimension we may solve the forward problem
\eqref{e:perm} on $D=(0,1)$, with $p(0) = p^-$ and $p(1) = p^+$
explicitly to obtain
\begin{equation}
  \label{e:perm_fwd_soln}
  p(x;u) = (p^+ - p^-) \frac{J_x(u)}{J_1(u)} + p^-, \quad J_x(u) \equiv
  \int_0^x \exp(-u(z))dz,
\end{equation}
and
\begin{equation}
  \Phi(u) = \frac{1}{2\gamma^2}\sum_{j=1}^\ell |p(x_j;u) - y_j|^2
\end{equation}
Following the methodology of \cite{Hinze:2009aa}, to compute
$D_u\Phi(u)$, we must solve the adjoint problem for $q$:
\begin{equation}
  \label{e:perm_adj}
  -\frac{d}{dx}\left(\exp(u) \frac{dq}{dx} \right)= -
  \frac{1}{\gamma^2}\sum_{j=1}^\ell (p(x_j;u) - y_j)\delta_{x_j}, \quad
  q(0) = q(1)=0.
\end{equation}
Again, we can write the solution explicitly via quadrature:
\begin{equation}
  \label{e:perm_adj_soln}
  \begin{split}
    q(x;u) &= K_x(u) - \frac{K_1(u) J_x(u)}{J_1(u)}, \\
    K_x(u)&\equiv \sum_{j=1}^\ell \frac{p(x_j;u) - y_j}{\gamma^2}
    \int_{0}^x \exp(-u(z)) H(z-x_j)dz
  \end{split}
\end{equation}
Using \eqref{e:perm_fwd_soln} and \eqref{e:perm_adj_soln},
\begin{equation}
  D_u\Phi(u) = \exp(u) \frac{d p(x;u)}{dx}\frac{d q(x;u)}{dx}.
\end{equation}

For this application we use a finite rank approximation of the
covariance of the approximating measure $\nu$, as explained in
subsection \ref{ssec:FR}.  In computing with the finite rank matrix
\eqref{e:lowrank}, it is useful, for good convergence, to work with $B
= \boldsymbol{\gamma}^{-1/2}$.  The preconditioned derivatives,
\eqref{e:preDJ_Gauss}, also require $D_B\Delta_0$, where $\Delta_0$ is
given by \eqref{e:D0}.  To characterize this term, if $v = \sum_i v_i
e_i$, we let ${\bf v} = (v_1, \ldots v_N)^T$ be the first $N$
coefficients.  Then for the finite rank approximation,
\begin{equation}
  \Phi_{\nu_0}(v) = \frac{1}{2}\left\langle v,(C^{-1} -
    C_0^{-1})v\right\rangle=\frac{1}{2} {\bf v}^T (\boldsymbol{\gamma} -
  \diag(\lambda_1^{-1},\ldots \lambda_N^{-1})) {\bf v}.
\end{equation}
Then using our parameterization with respect to the matrix $B$,
\begin{equation}
  D_B \Delta_0(v) =D_B (\Phi(m+v) - \Phi_{\nu_0}(v)) =
  \frac{1}{2}\left[B^{-1} {\bf v} ( B^{-2} {\bf v})^T  +B^{-2} {\bf v} ( B^{-1} {\bf v})^T\right].
\end{equation}
As a preconditioner for \eqref{e:preDJ_Gauss_C} we found that it was
sufficient to multiply by $\lambda_N$.

We solve this problem with Ranks $K=$ 2, 4, 6, first minimizing $\Dkl$,
and then running the pCN Algorithm \ref{a2} to sample from $\mu_y$.
The common parameters are:
\begin{itemize}
\item $\gamma = 0.1$, $\delta=1$, $p^{-}=0$ and $p^+=2$;
\item There are $2^7$ uniformly spaced grid points in $[0,1)$;
\item \eqref{e:perm_fwd_soln} and \eqref{e:perm_adj_soln} are solved
  via trapezoidal rule quadrature;
\item The true value of $u(x) = 2\sin(2\pi x)$;
\item The dimension of the data is four, with samples at $x=0.2, 0.4,
  0.6, 0.8$;
\item $m_0= 0$ and $B_0 = \diag(\lambda_n)$, $n \leq \text{Rank}$;
\item $\int \dot m^2$ is estimated spectrally;
\item $10^5$ iterations of the Robbins-Monro algorithm are performed
  with $10^2$ samples per iteration;
\item $a_0 = .1$ and $a_n = a_0 n^{-3/5}$;
\item The eigenvalues of $\sigma$ are constrained to the interval
  $[10^{-4}, 10^0]$ and the mean is constrained to $[-5,5]$;
\item pCN Algorithms \ref{a1} and \ref{a2} are implemented with $\beta
  =0.6$, and $10^6$ iterations.
\end{itemize}

The results of the $\Dkl$ optimization phase of the problem, using the
Robbins-Monro Algorithm \ref{a:RMforKL}, appear in Figure
\ref{f:bayesinverse_kl}.  This figure shows: the convergence of $m_n$ in the Rank 2
case; the convergence of the eigenvalues of $B$ for Ranks 2, 4, and 6; and the minimization of
$\Dkl$.  We only present the convergence of the mean in the
Rank 2 case, as the others are quite similar.  At the termination of
the Robbins-Monro step, the $B_n$ matrices are:
\begin{align}
  {B}_n &= \begin{pmatrix}
    0.0857  &  0.00632\\
    - & 0.105
  \end{pmatrix}\\
  {B}_n &= \begin{pmatrix}
    0.0864  &   0.00500  &  -0.00791  & -0.00485\\
    -              &    0.106   &  0.00449   &-0.00136\\
    -               & -                    &     0.0699  &-0.000465\\
    - & - & - & 0.0739
  \end{pmatrix}\\
  {B}_n &= \begin{pmatrix}
    0.0870  &  0.00518 &   -0.00782 &   -0.00500&   -0.00179&   -0.00142\\
    - &      0.106&    0.00446 &   -0.00135&     0.00107&    0.00166\\
    - & - & 0.0701 & -0.000453 & -0.00244&
    9.81\times 10^{-5}\\
    - &  - &  - &     0.0740 &   -0.00160&    0.00120\\
    -  &   - &  - &  - & 0.0519 &   -0.00134\\
    - & - & - & - & -& 0.0523
  \end{pmatrix}
\end{align}
Note there is consistency as the rank increases, and this is reflected
in the eigenvalues of the $B_n$ shown in Figure
\ref{f:bayesinverse_kl}.  As in the case of the scalar problem, more
iterations of Robbins-Monro are computed than are needed to ensure
convergence.

\begin{figure}

  \begin{center}
    \subfigure[Convergence of $m_n(x)$ for Rank
    2]{\includegraphics[width=6.25cm]{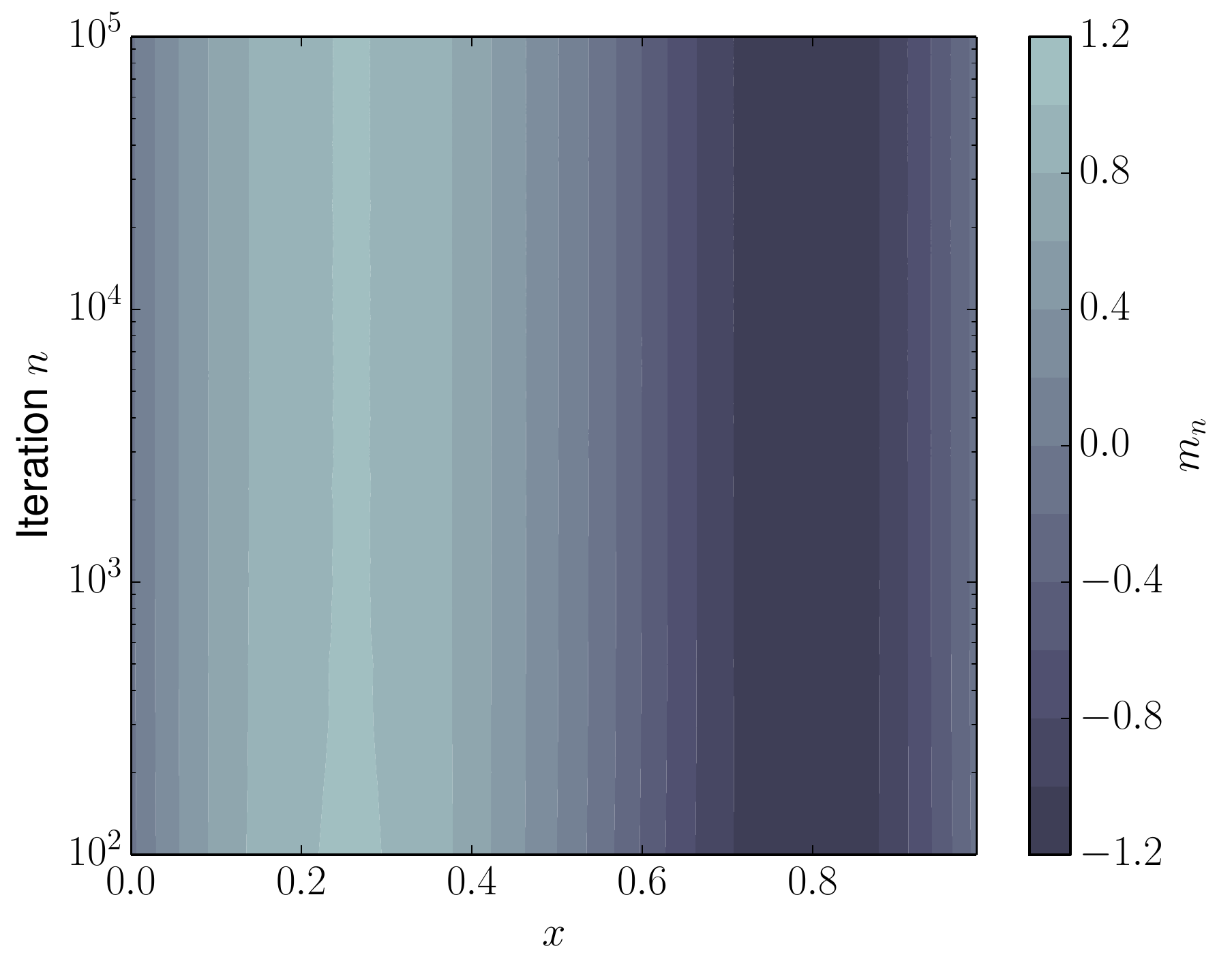}}
    \subfigure[$m_n(x)$ for Rank 2 at Particular
    Iterations]{\includegraphics[width=6.25cm]{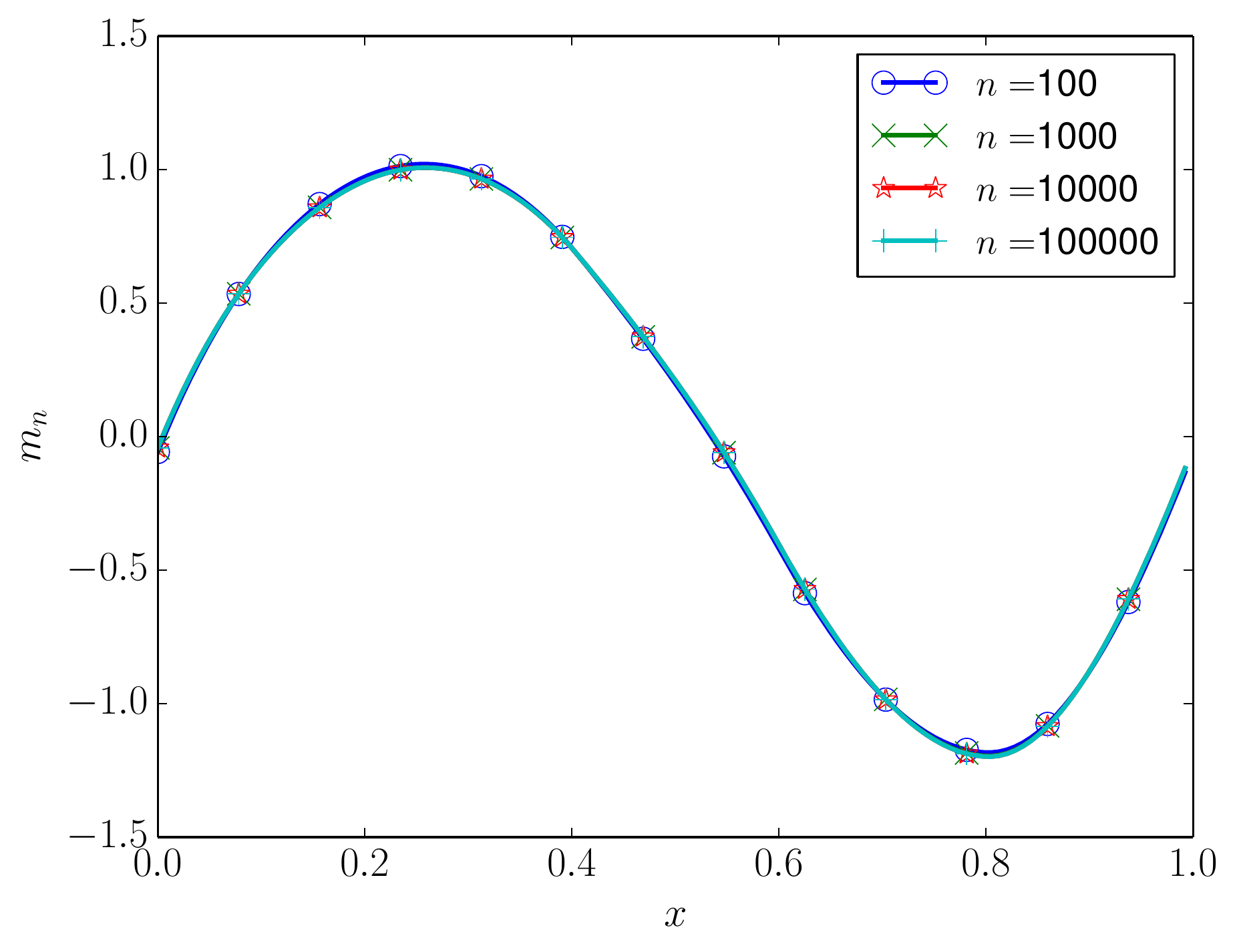}}

    \subfigure[Convergence of the eigenvalues of
    $B_n$]{\includegraphics[width=6.25cm]{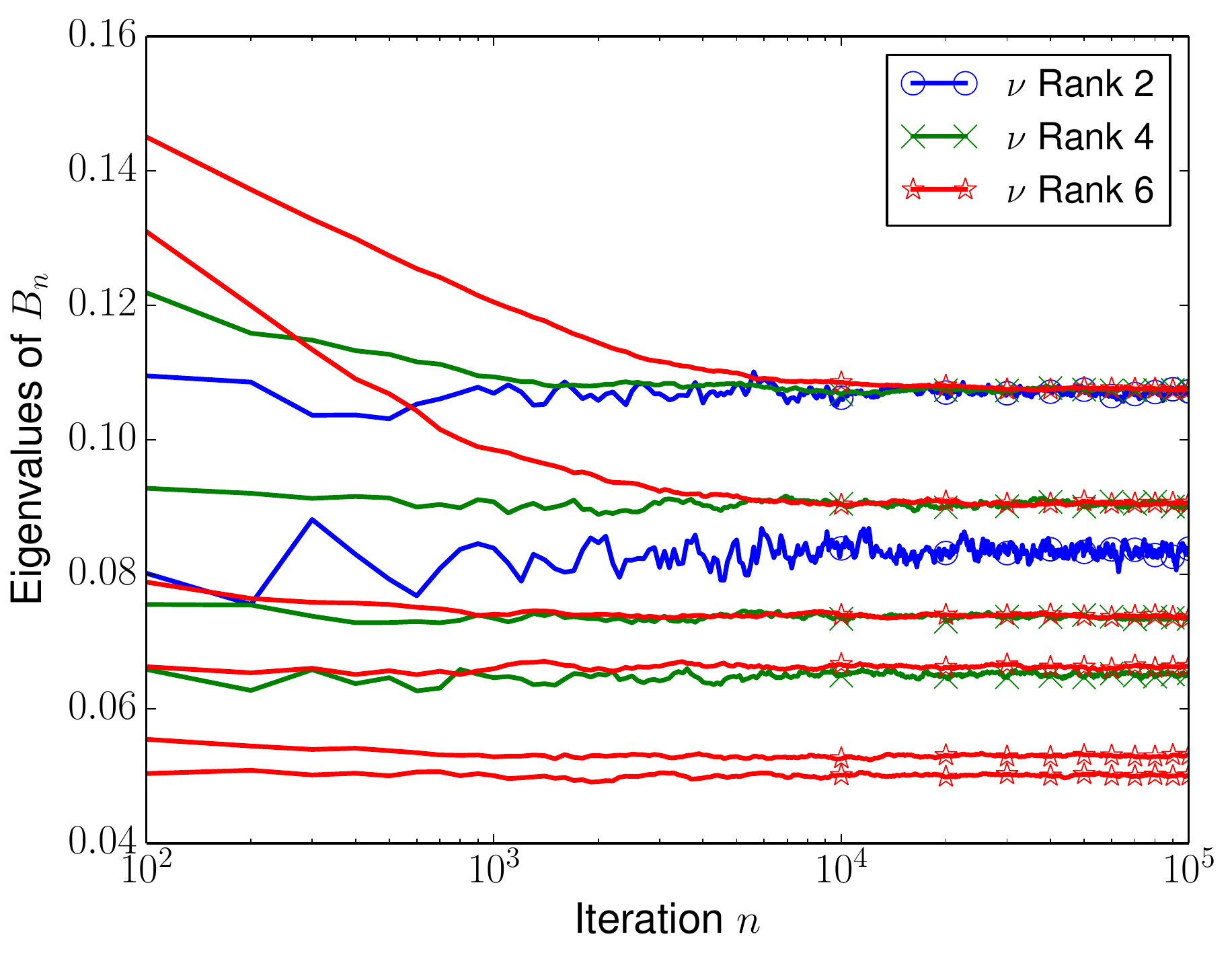}}
    \subfigure[Minimization of
    $\Dkl$]{\includegraphics[width=6.25cm]{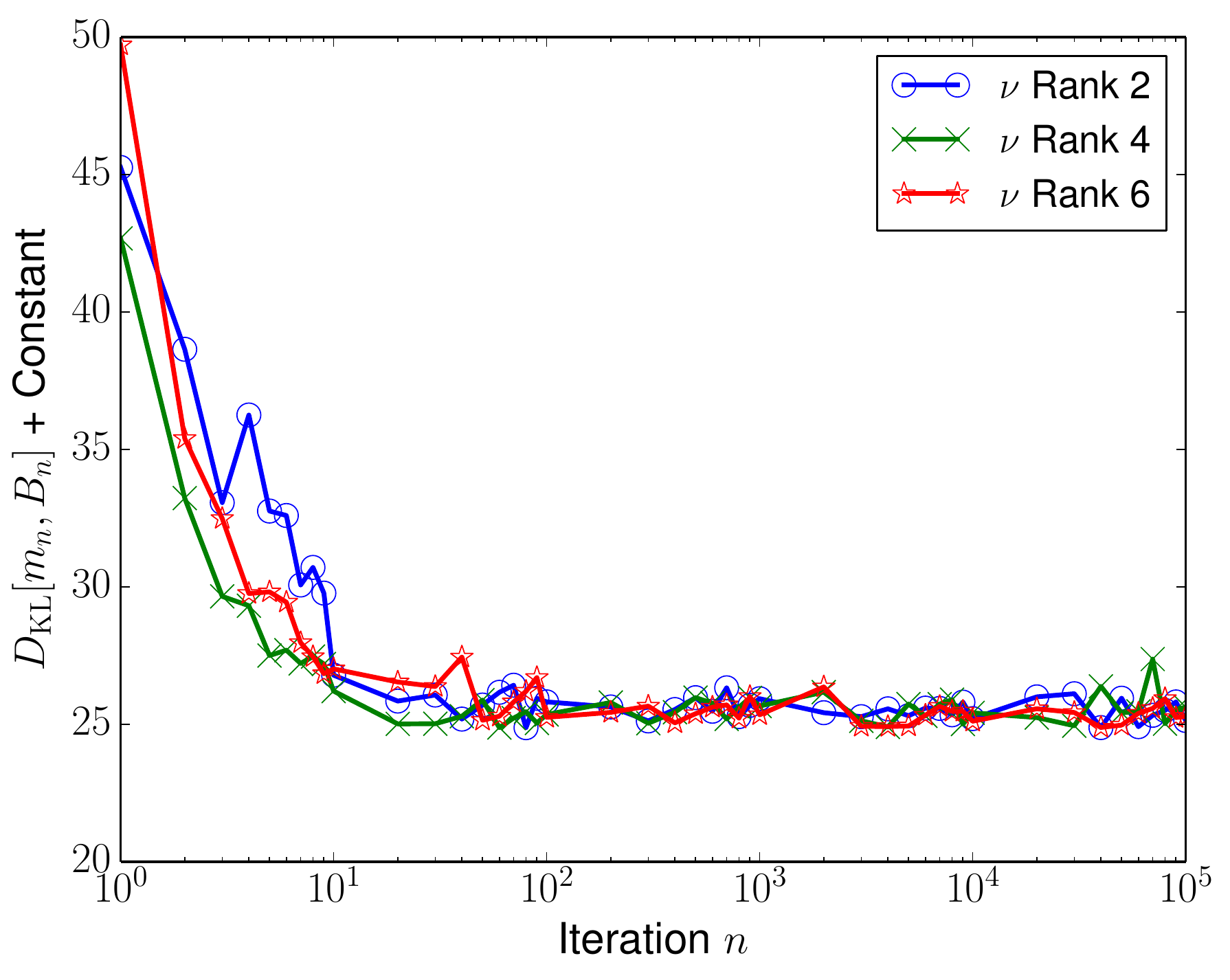}}

  \end{center}

  \caption{Convergence of the Robbins-Monro Algorithm \ref{a:RMforKL}
    applied to the Bayesian Inverse problem. Figures (a) and (b) show the
    convergence of $m_n$ in the case of Rank 2, while Figure (c) shows
    the convergence of the eigenvalues of $B_n$ for Ranks 2, 4 and
    6. Figure (d) shows the minimization of $\Dkl$. The observational
    noise is $\gamma = 0.1$.  The figures indicate that Rank 2 has
    converged after $10^2$ iterations; Rank 4 has converged after
    $10^3$ iterations; and Rank 6 has converged after $10^4$
    iterations.}

  \label{f:bayesinverse_kl}

\end{figure}

The posterior sampling, by means of Algorithms \ref{a1} and \ref{a2},
is described in Figure~\ref{f:bayesinverse_post}.  There is good
posterior agreement in the means and variances in all cases, and the
low rank priors provide not just good means but also variances.  This
is reflected in the high acceptance rates and low auto covariances;
there is approximately an order of magnitude in improvement in using
Algorithm \ref{a2}, which is informed by the best Gaussian
approximation, and Algorithm \ref{a1}, which is not.

However, notice in Figure~\ref{f:bayesinverse_kl} that the posterior,
even when $\pm$ one standard deviation is included, does not capture
the truth.  The results are more favorable when we consider the
pressure field, and this hints at the origin of the disagreement.  The
values at $x=0.2$ and 0.4, and to a lesser extent at 0.6, are
dominated by the noise.  Our posterior estimates reflect the
limitations of what we are able to predict given our assumptions.  If
we repeat the experiment with smaller observational noise, $\gamma =
0.01$ instead of $0.1$, we see better agreement, and also variation in
performance with respect to approximations of different ranks.  These
results appear in Figure \ref{f:bayesinverse_post01}.  In this smaller
noise case, there is a two order magnitude improvement in performance.

\begin{figure}
  \begin{center}
    \subfigure[Log permeability
    $u(x)$]{\includegraphics[width=6.25cm]{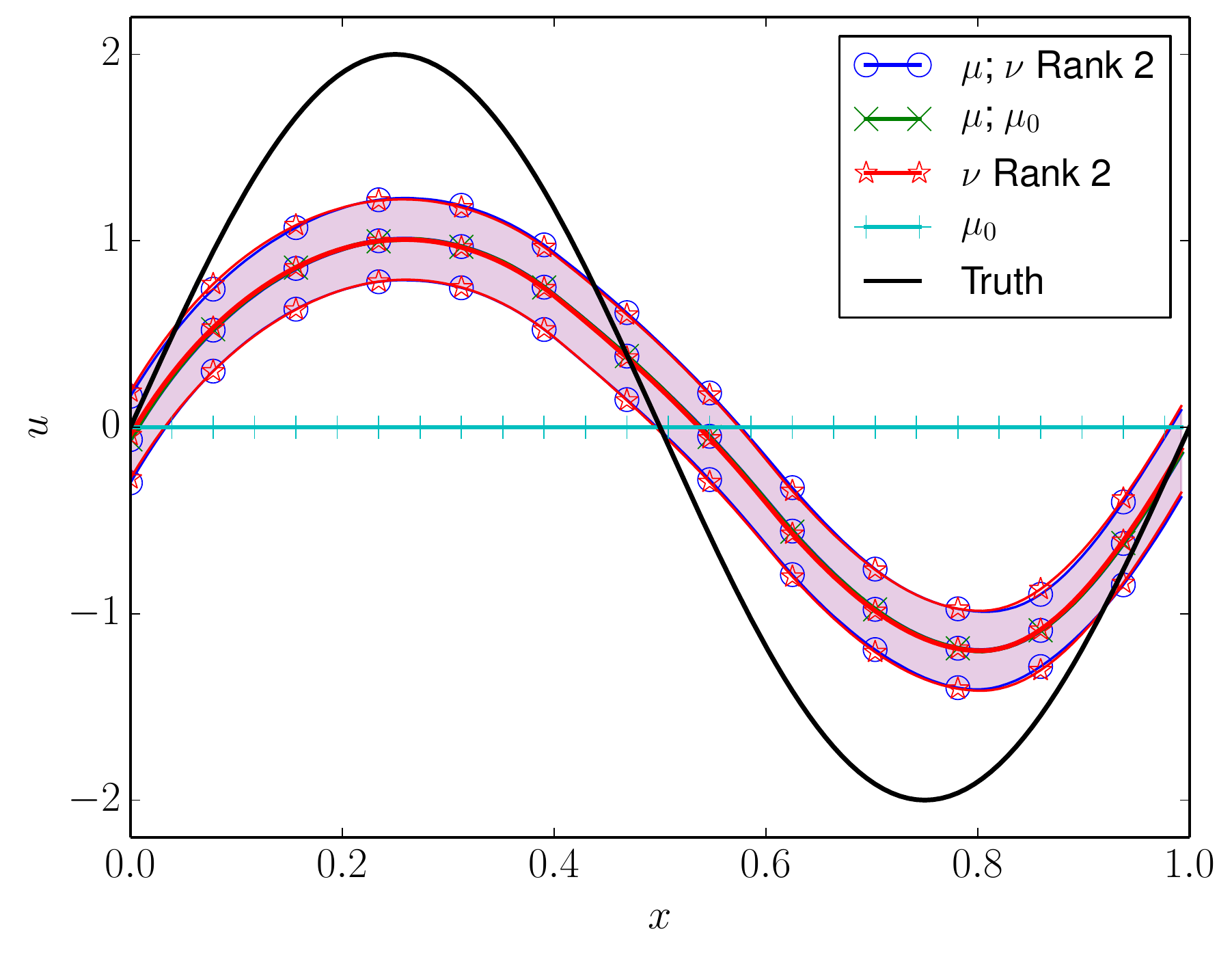}}
    \subfigure[Pressure
    $p(x)$]{\includegraphics[width=6.25cm]{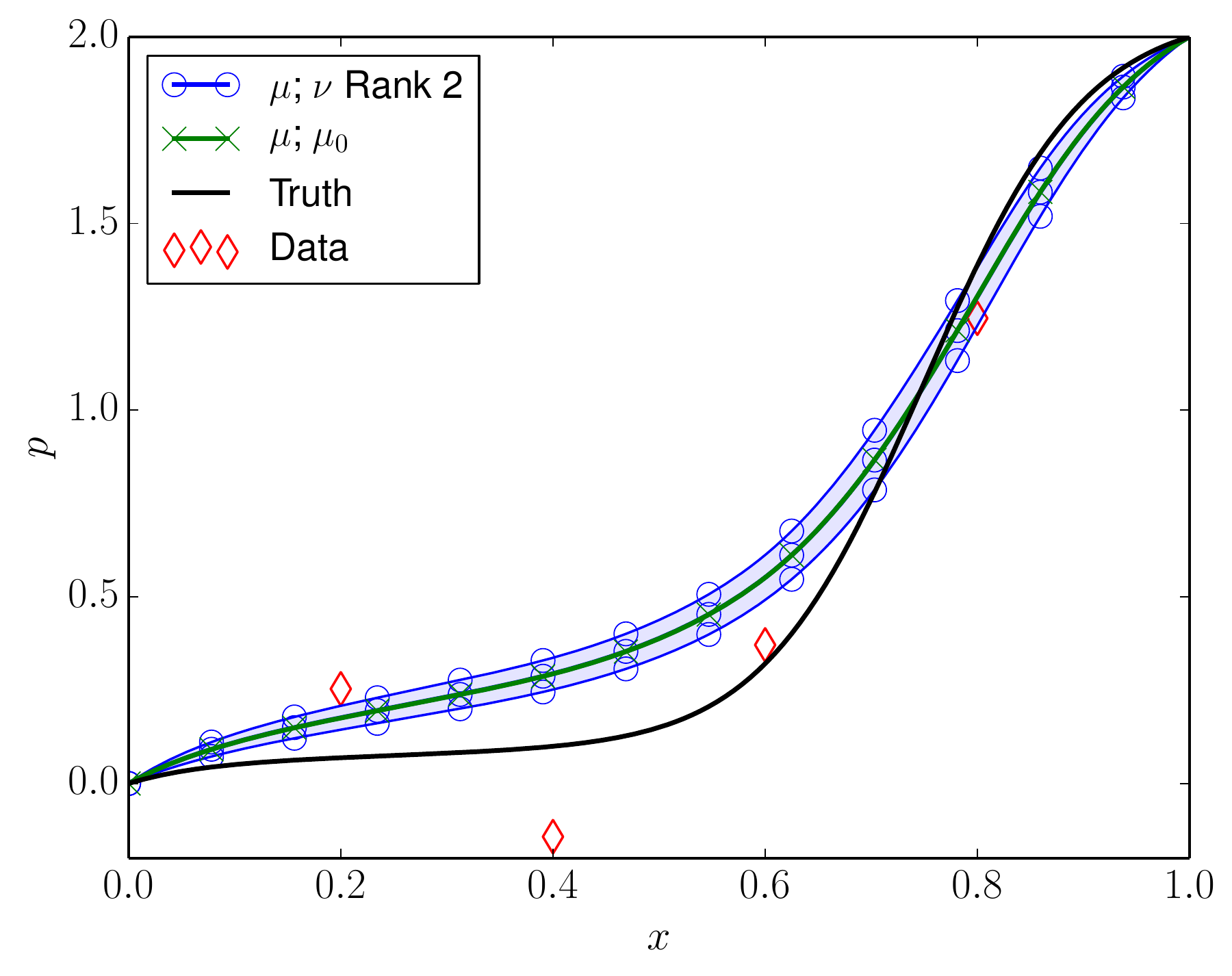}}

    \subfigure[Acceptance
    Rate]{\includegraphics[width=6.25cm]{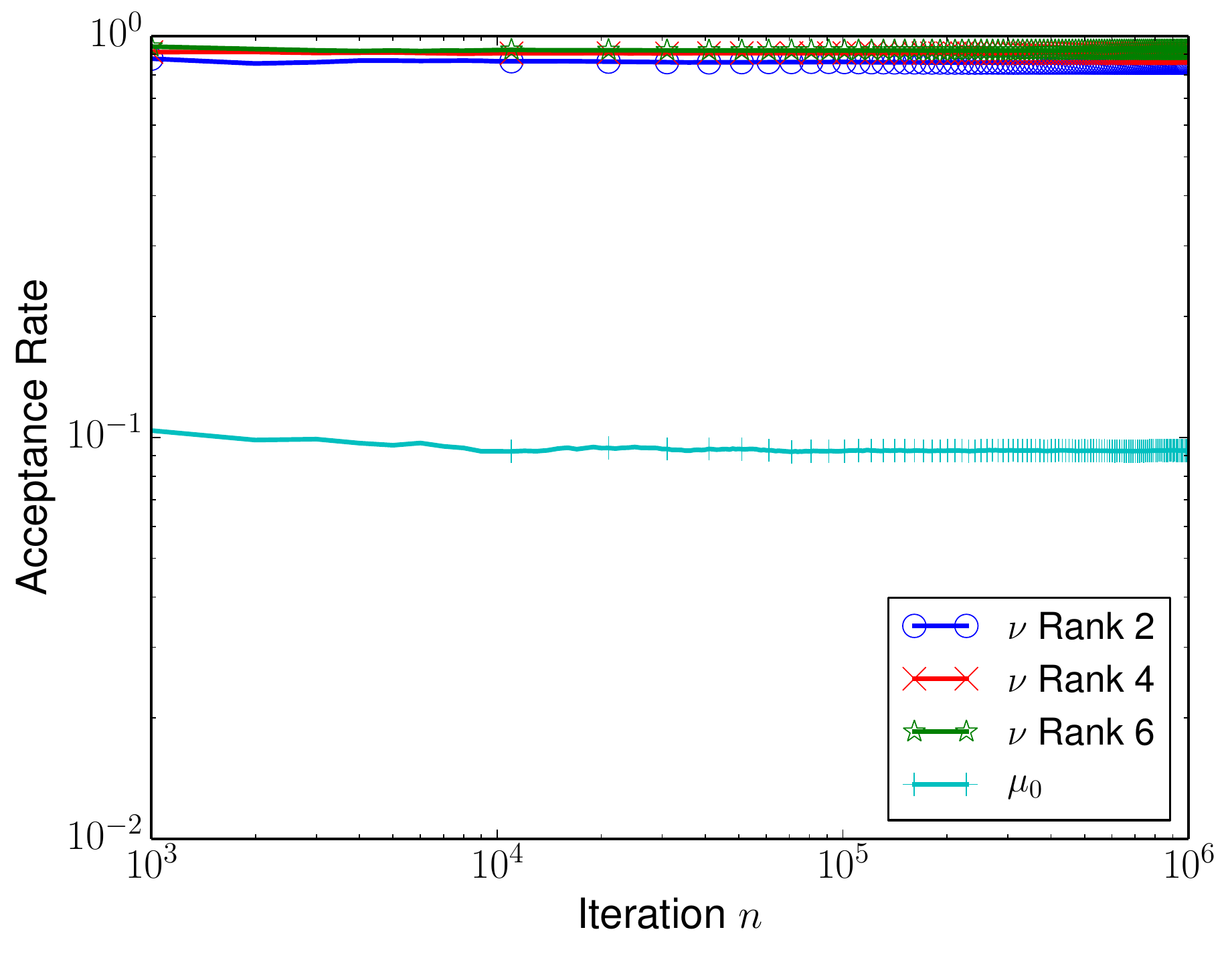}}
    \subfigure[Autocovariance]{\includegraphics[width=6.25cm]{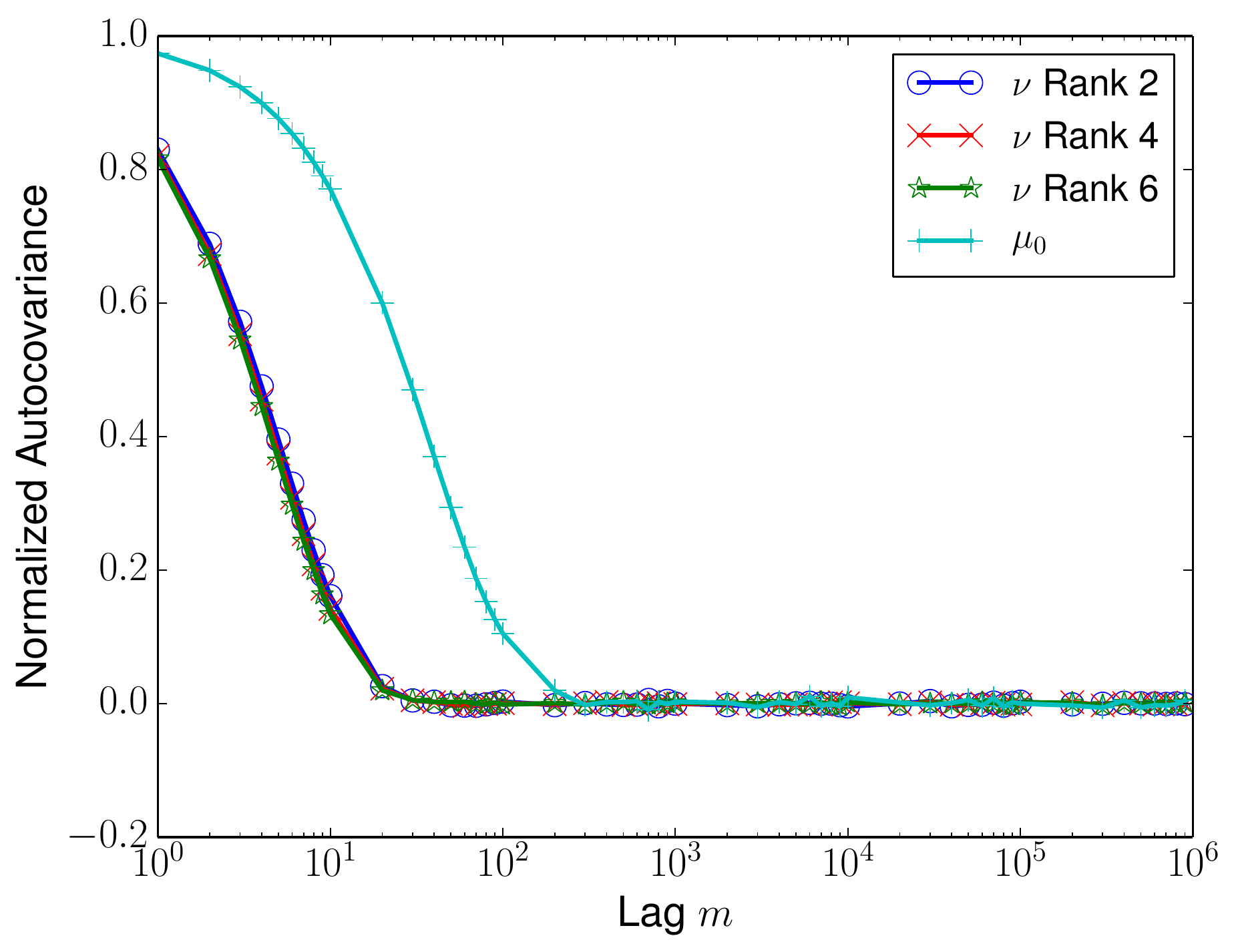}}

  \end{center}

  \caption{Behavior of MCMC Algorithms \ref{a1} and \ref{a2} for the
    Bayesian Inverse problem with observational noise $\gamma = 0.1$.
    The true posterior distribution, $\mu$, is sampled using $\mu_0$
    (Algorithm \ref{a1}) and $\nu$, with Ranks 2, 4 and 6 (Algorithm
    \ref{a2}). The resulting posterior approximations are labeled
    $\mu;\;\mu_0$ (Algorithm \ref{a1}) and $\mu;\;\nu$ Rank 2,
    (Algorithm \ref{a2}).  The notation $\mu_0$ and $\nu$ Rank $K$ is
    used for the prior and best Gaussian approximations of the
    corresponding rank.  The distributions of $u(x)$, in Figure (a),
    for the optimized $\nu$ Rank 2 and the posterior $\mu$ overlap,
    but are still far from the truth.  The results for Ranks 4 and 6
    are similar.  Figures (c) and (d) compare the performance of
    Algorithm \ref{a2} when using $\nu$ Rank K for the proposal, with
    $K=$2, 4, and 6, against Algorithm \ref{a1}.  $\nu$ Rank 2 gives an
    order of magnitude improvement in posterior sampling over $\mu_0$.
    There is not significant improvement when using $\nu$ Ranks 4 and
    6 over using Rank 2.  Shaded regions enclose $\pm$ one standard
    deviation.}
  \label{f:bayesinverse_post}
\end{figure}

\begin{figure}
  \begin{center}
    \subfigure[Log permeability
    $u(x)$]{\includegraphics[width=6.25cm]{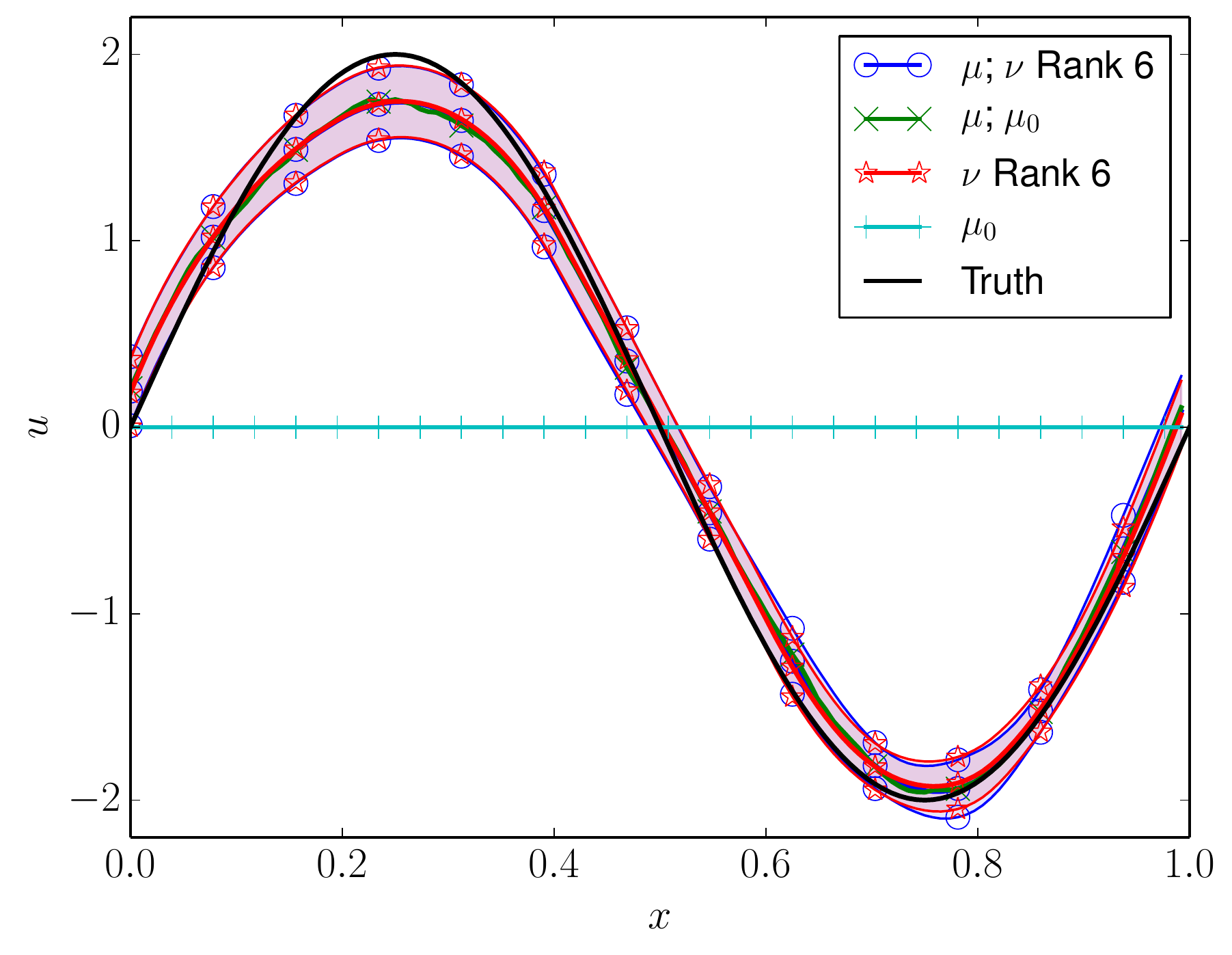}}
    \subfigure[Pressure
    $p(x)$]{\includegraphics[width=6.25cm]{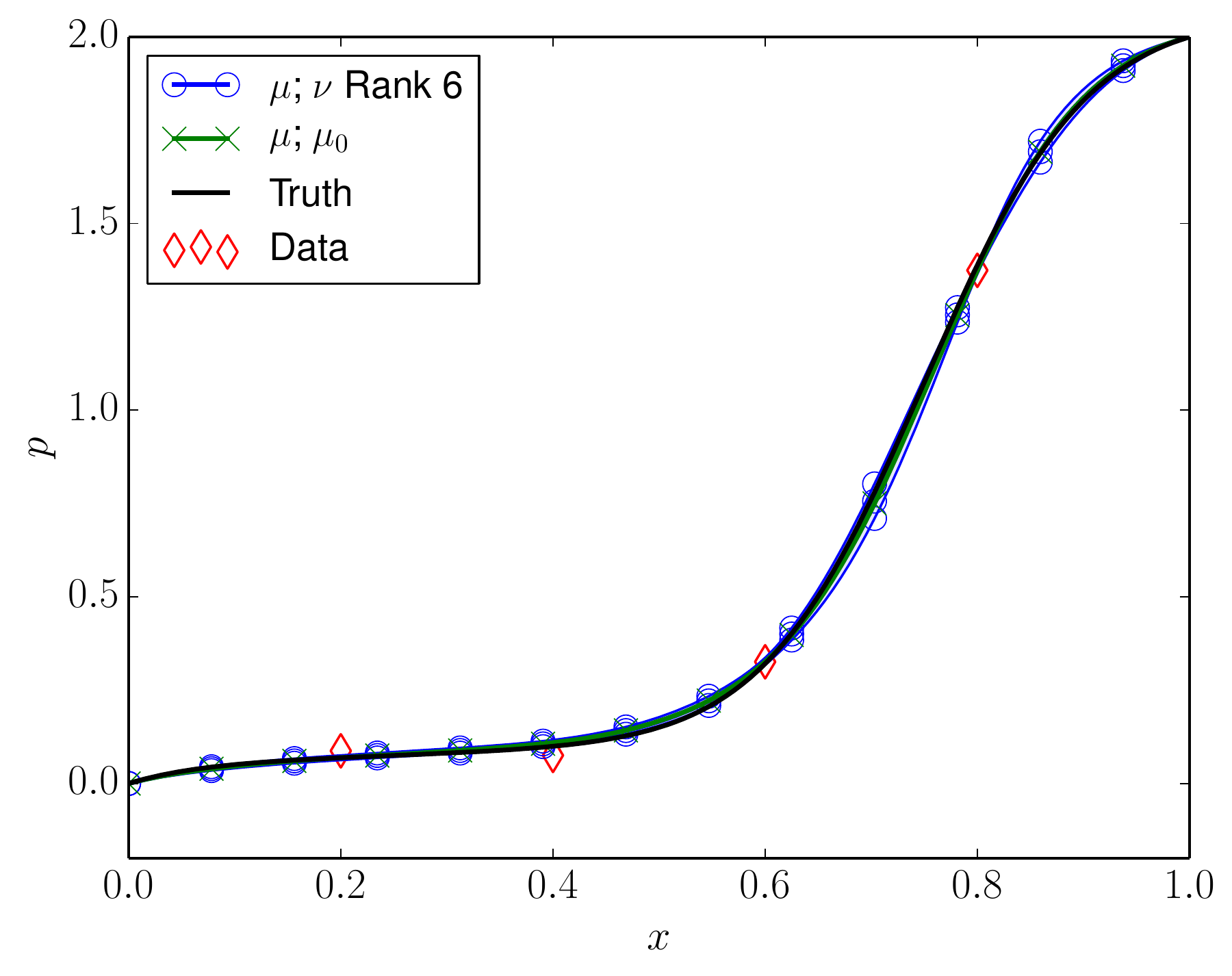}}

    \subfigure[Acceptance
    Rate]{\includegraphics[width=6.25cm]{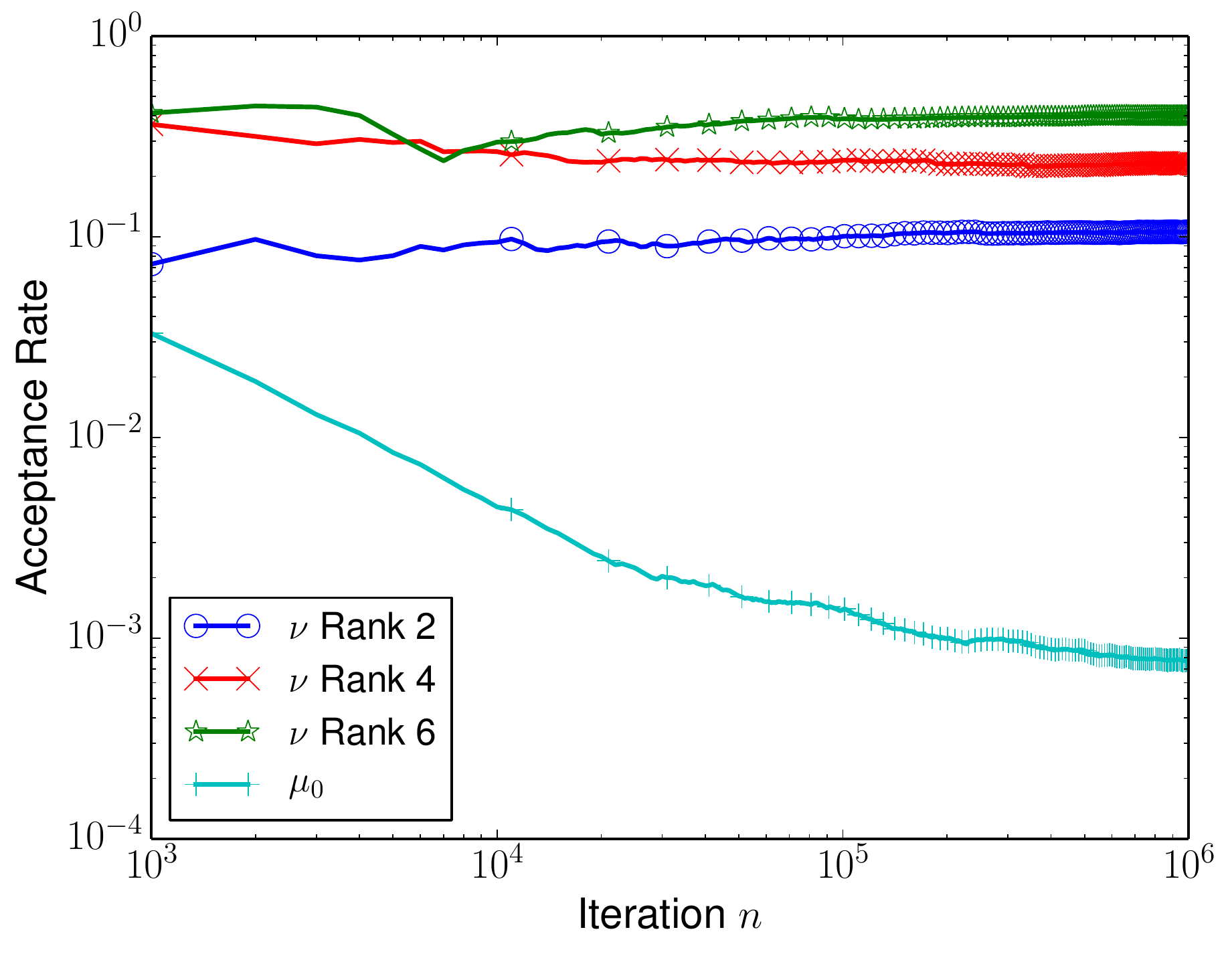}}
    \subfigure[Autocovariance]{\includegraphics[width=6.25cm]{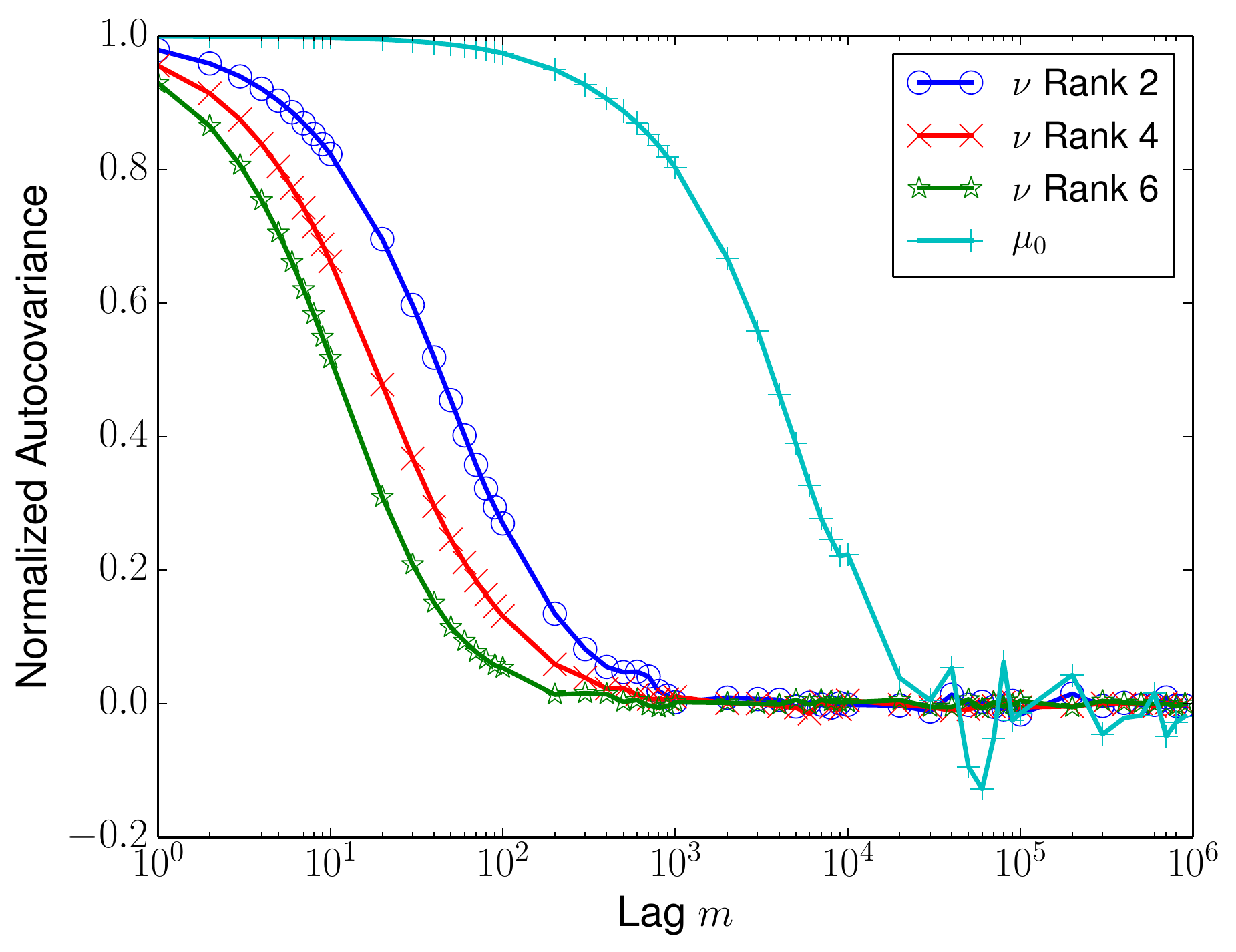}}

  \end{center}

  \caption{Behavior of MCMC Algorithms \ref{a1} and \ref{a2} for the
    Bayesian Inverse problem with observational noise $\gamma =
    0.01$. Notation as in Figure \ref{f:bayesinverse_post}.  The
    distribution of $u(x)$, shown in Figure (a), for both the optimized Rank 6 $\nu$, and
    the posterior $\mu$ overlap, and are close to the truth.  Unlike
    the case of $\gamma=0.1$, Figures (c) and (d) show improvement in
    using $\nu$ Rank 6 within Algorithm \ref{a2}, over Ranks 2 and 4.
    However, all three cases of Algorithm \ref{a2} are at
    least two orders of magnitude better than Algorithm \ref{a1},
    which uses only $\mu_0$.  Shaded regions enclose $\pm$ one
    standard deviation.}
  \label{f:bayesinverse_post01}
\end{figure}

\subsection{Conditioned Diffusion Process}
\label{ssec:CDP}

Next, we consider measure $\mu$ given by \eqref{e:target} in the case
where $\mu_0$ is a unit Brownian bridge connecting $0$ to $1$ on the
interval $(0,1)$, and
$$\Phi = \frac{1}{4\epsilon^2}\int_0^1 \Bigl(1-u(t)^2\Bigr)^2 dt,$$
a double well potential. This also has an interpretation as a
conditioned diffusion \cite{reznikoff2005invariant}.  Note that $m_0 =
t$ and $C_0^{-1} = -\tfrac{1}{2}\tfrac{d^2}{dt^2}$ with
$D(C_0^{-1})=H^2(I)\cap H^1_0(I)$ with $I=(0,1).$

We seek the approximating measure $\nu$ in the form $N(m(t),C)$ with
$(m,B)$ to be varied, where
\[
C^{-1} = C_0^{-1} + \tfrac{1}{2\epsilon^2}B
\]
and $B$ is either constant,$B \in \R$, or $B:I \to \R$ is a function
viewed as a multiplication operator.

We examine both cases of this problem, performing the optimization,
followed by pCN sampling.  The results were then compared against the
uninformed prior, $\mu_0 = N(m_0, C_0)$.  For the constant $B$ case,
no preconditioning on $B$ was performed, and the initial guess was $B
= 1$. For $B = B(t)$, a Tikhonov-Phillips regularization was
introduced,
\begin{equation}
  \Dkl^\alpha = \Dkl + \frac{\alpha}{2}\int \dot B^2 dt, \quad \alpha = 10^{-2}.
\end{equation}

For computing the gradients \eqref{e:preDJ_Gauss} and estimating
$\Dkl$,
\begin{subequations}
  \begin{align}
    D_m\Phi(v+m) &= \tfrac{1}{2\eps^2} (v+m)[(v+m)^2-1], \\
    \label{e:516b}
    D_B\Phi_{\nu_0}(v) &= \begin{cases}
      \frac{1}{4\eps^2}\int_0^1 v^2 dt & \text{$B$ constant}\\
      \frac{1}{4\eps^2} v^2 &\text{$B(t)$}
    \end{cases}.
  \end{align}
\end{subequations}
No preconditioning is applied for \eqref{e:516b} in the case that $B$
is a constant, while in the case that $B(t)$ is variable, the
preconditioned gradient in $B$ is
\[
\left\{ -\alpha\tfrac{d^2}{dt^2}\right\}^{-1}\left(\E^{\nu_0}(\Delta_0
  D_\theta \Delta_0) - \E^{\nu_0}(\Delta_0) \E^{\nu_0}(D_\theta
  \Delta_0)\right)+ B.
\]
Because of the regularization, we must invert $-d^2/dt^2$, requiring
the specification of boundary conditions.  By a symmetry argument, we
specify the Neumann boundary condition, $B'(0) = 0$.  At the other
endpoint, we specify the Dirichlet condition $B(1)= V''(1)=2$, a ``far
field'' approximation.  

The common parameters used are:
\begin{itemize}
\item The temperature $\epsilon = 0.05$;
\item There were $99$ uniformly spaced grid points in $(0,1)$;
\item As the endpoints of the mean path are 0 and 1, we constrained
  our paths to lie in $[0, 1.5]$;
\item $B$ and $B(t)$ were constrained to lie in $[10^{-3}, 10^1]$, to
  ensure positivity of the spectrum;
\item The standard second order centered finite difference scheme was
  used for $C_0^{-1}$;
\item Trapezoidal rule quadrature was used to estimate $\int_0^1 \dot
  m^2$ and $\int_0^1 \dot B^2 dt$, with second order centered
  differences used to estimate the derivatives;
\item $m_0(t) = t$, $B_0 = 1$, $B_0(t) =V''(1) $, the right endpoint
  value;
\item $10^5$ iterations of the Robbins-Monro algorithm are performed
  with $10^2$ samples per iteration;
\item $a_0=2$ and $a_n = a_0 n^{-3/5}$;
\item pCN Algorithms \ref{a1} and \ref{a2} are implemented with $\beta
  =0.6$, and $10^6$ iterations.
\end{itemize}
Our results are favorable, and the outcome of the Robbins-Monro
Algorithm \ref{a:RMforKL} is shown in Figures \ref{f:constB} and
\ref{f:varB} for the additive potentials $B$ and $B(t)$, respectively.
The means and potentials converge in both the constant and variable
cases.  Figure \ref{f:ac_dkl} confirms that in both cases, $\Dkl$ and
$\Dkl^\alpha$ are reduced during the algorithm.

\begin{figure}
  \begin{center}
    \subfigure[$m_n(t)$ versus
    $n$.]{\includegraphics[width=6.25cm]{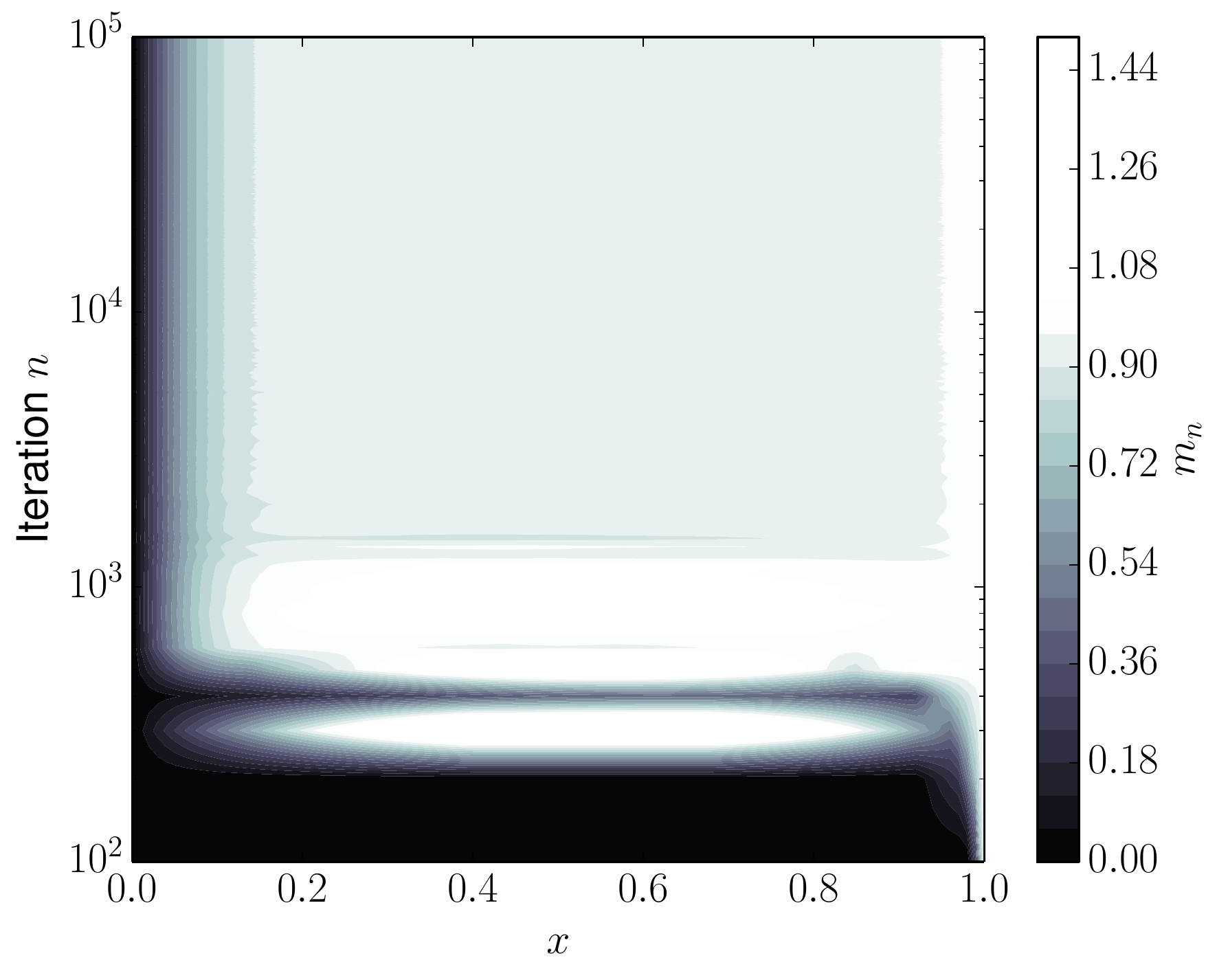}}
    \subfigure[$m_n(t)$ at Particular
    Iterations]{\includegraphics[width=6.25cm]{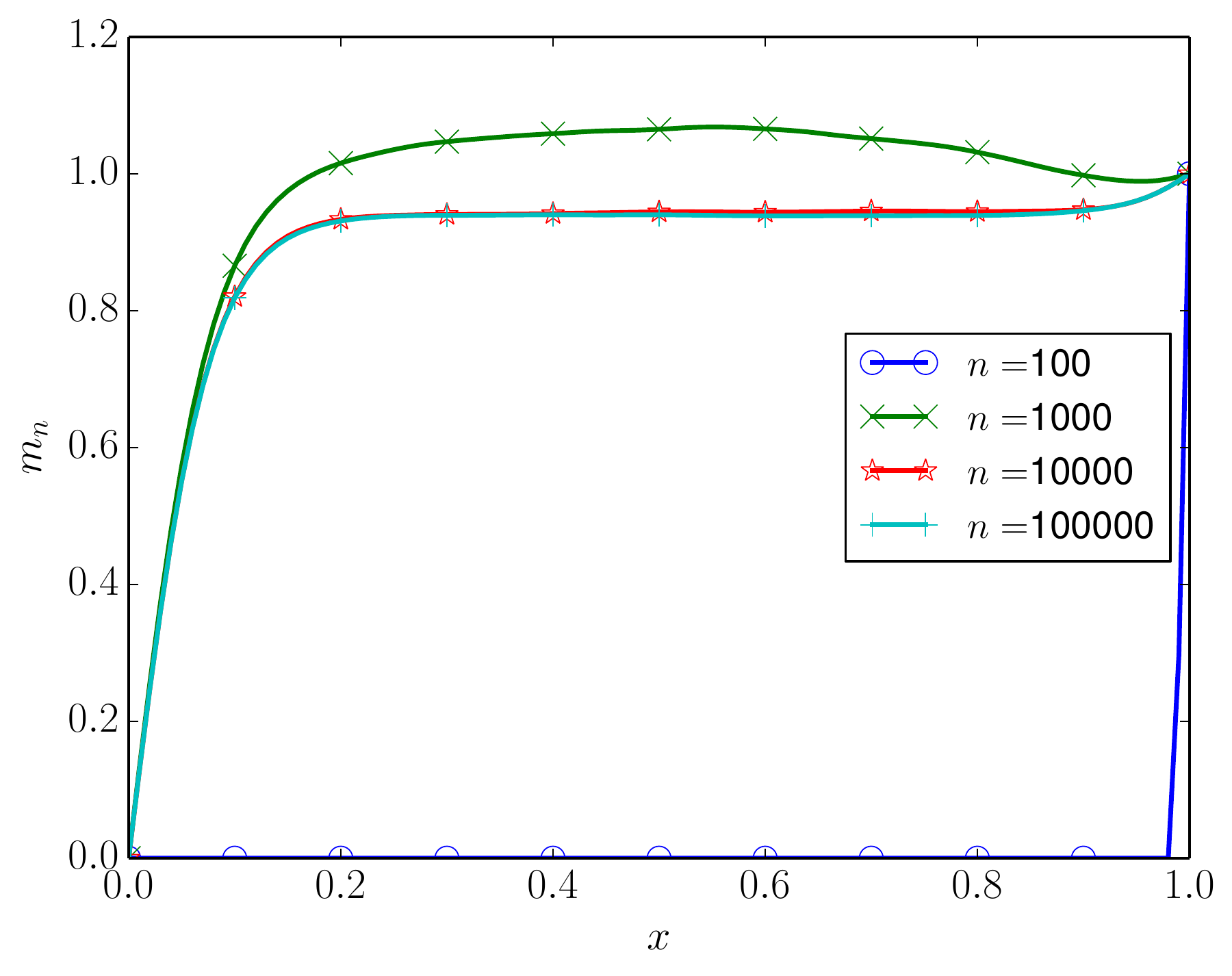}}

    \subfigure[$B_n$ versus
    $n$.]{\includegraphics[width=6.25cm]{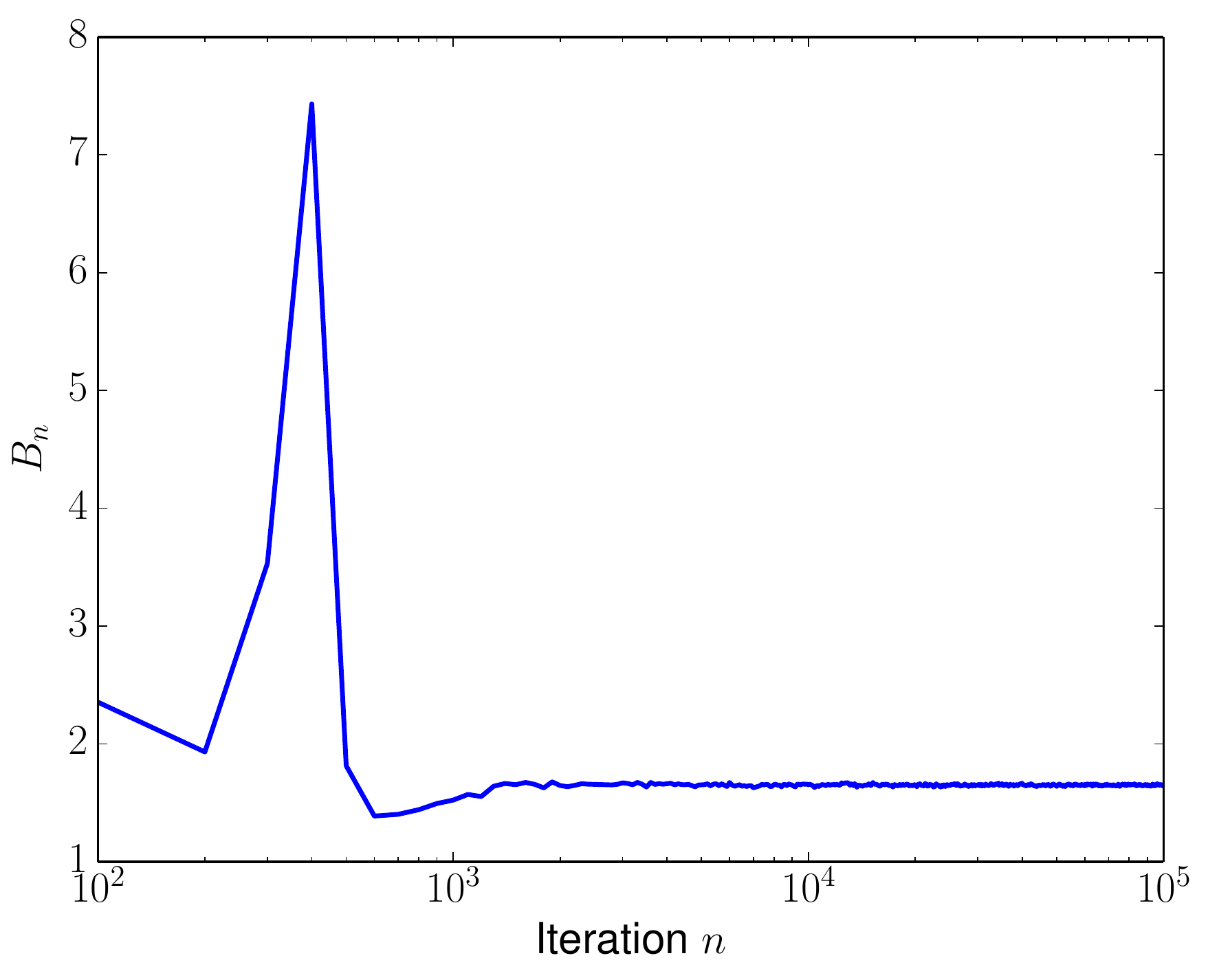}}
  \end{center}

  \caption{Convergence of the Robbins-Monro Algorithm \ref{a:RMforKL}
    applied to the Conditioned Diffusion problem, in the case of
    constant inverse covariance potential $B$.  Figure (a) shows
    evolution of $m_n(t)$ with $n$; Figure (b) shows $m_n(t)$ at
    particular $n$.  Figure (c) shows convergence of the $B_n$ constant.}
  \label{f:constB}
\end{figure}

\begin{figure}
  \begin{center}

    \subfigure[$m_n(t)$ versus
    $n$.]{\includegraphics[width=6.25cm]{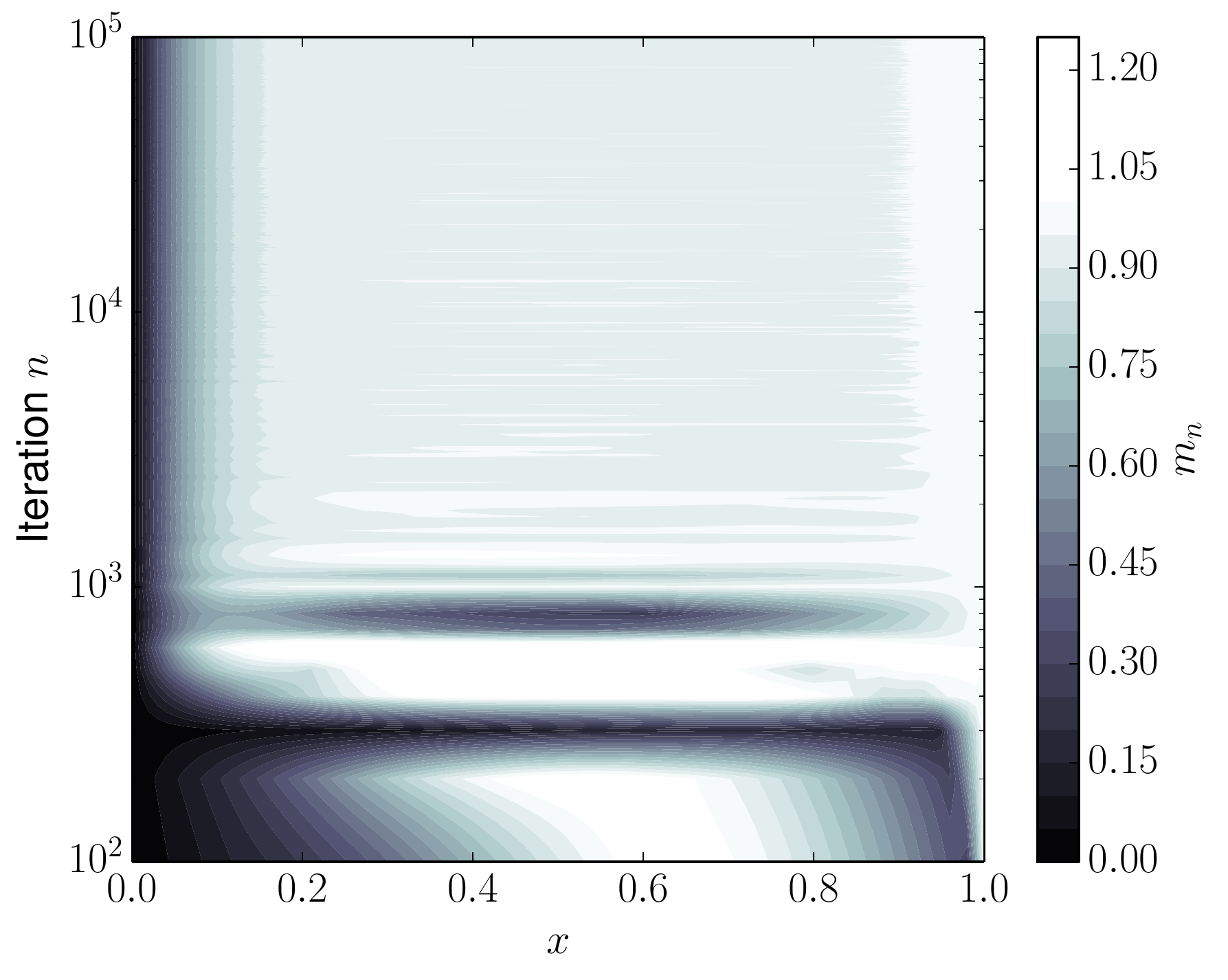}}
    \subfigure[$m_n(t)$ at Particular
    Iterations.]{\includegraphics[width=6.25cm]{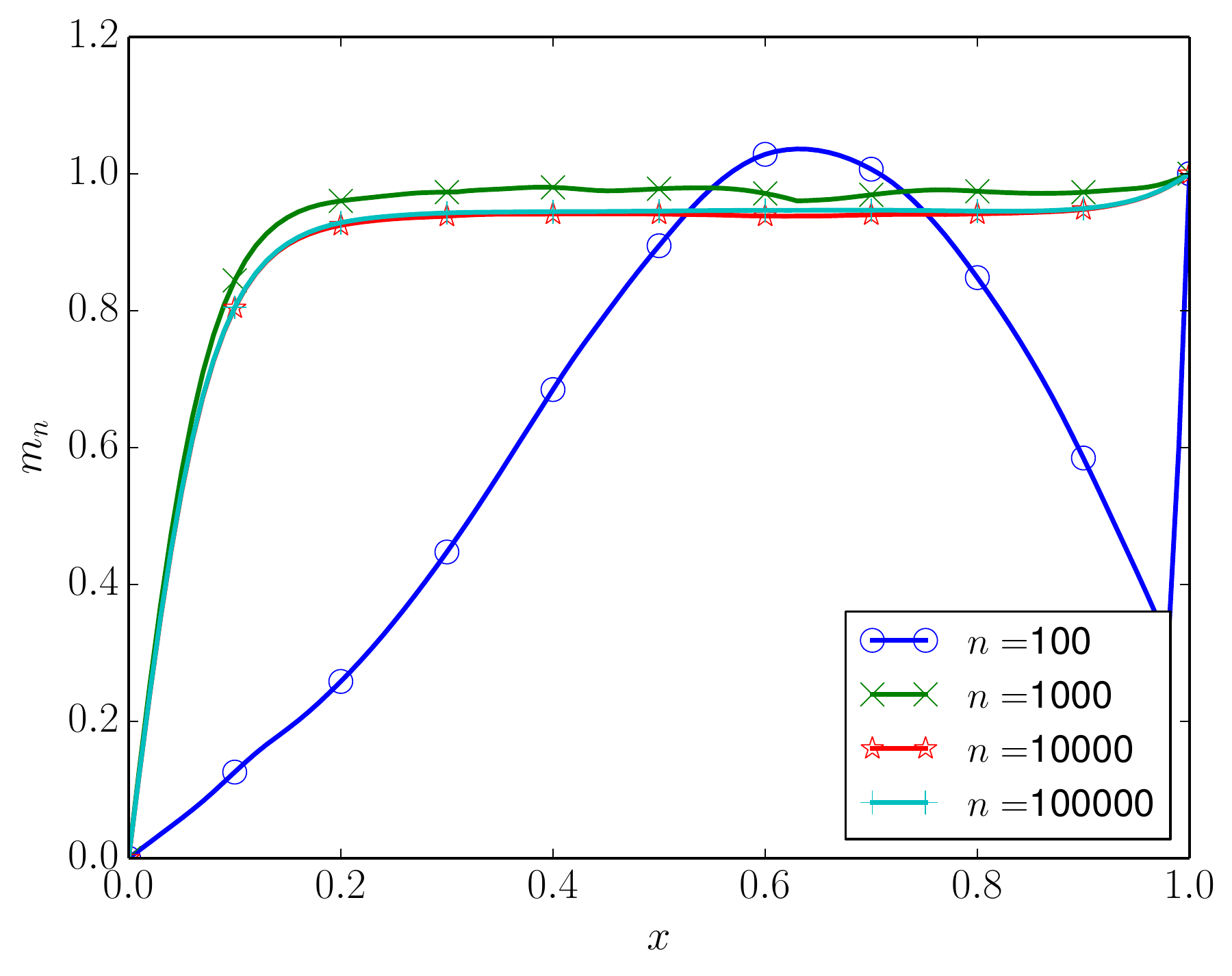}}

    \subfigure[$B_n(t)$ versus
    $n$.]{\includegraphics[width=6.25cm]{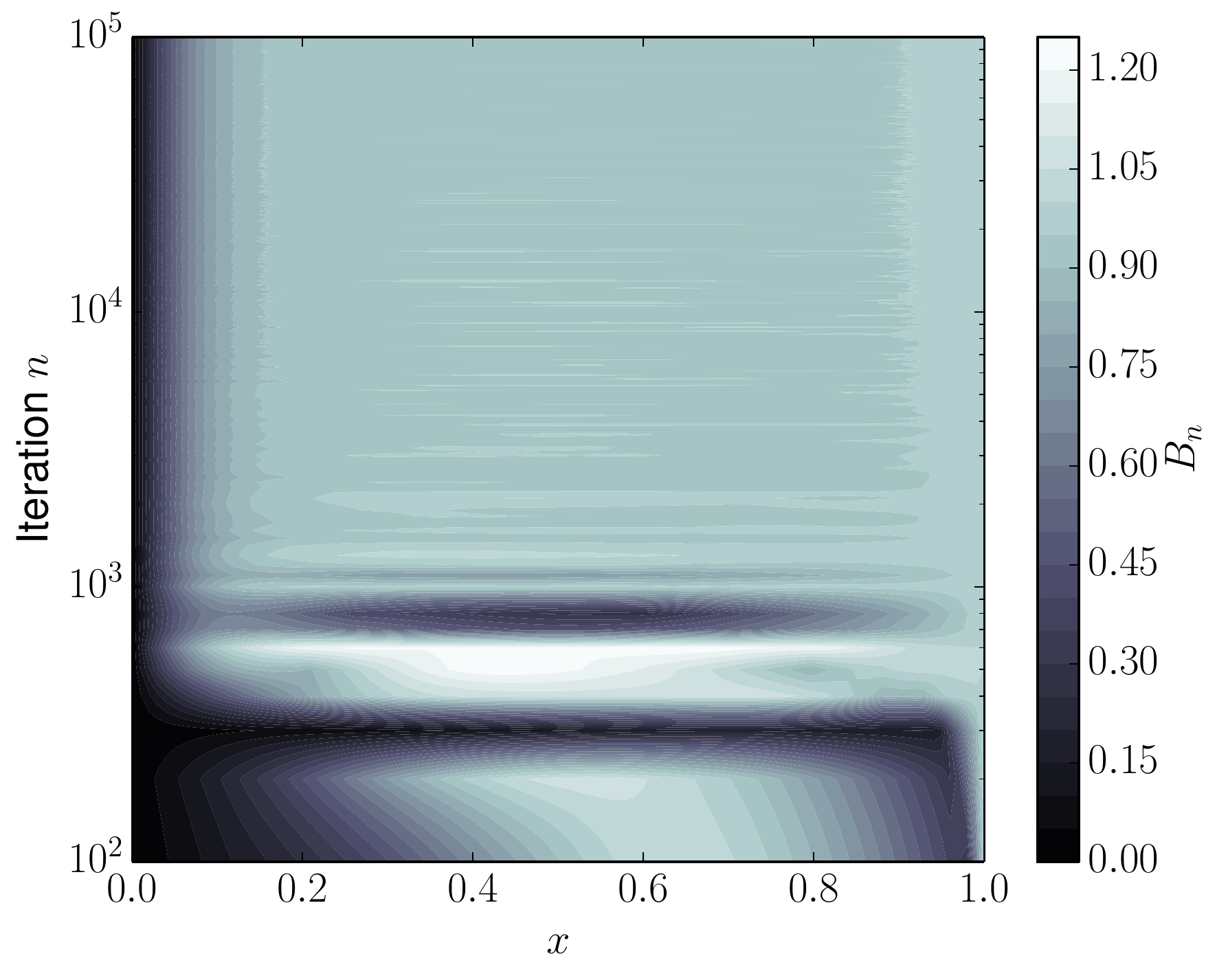}}
    \subfigure[$B_n(t)$ at Particular
    Iterations. ]{\includegraphics[width=6.25cm]{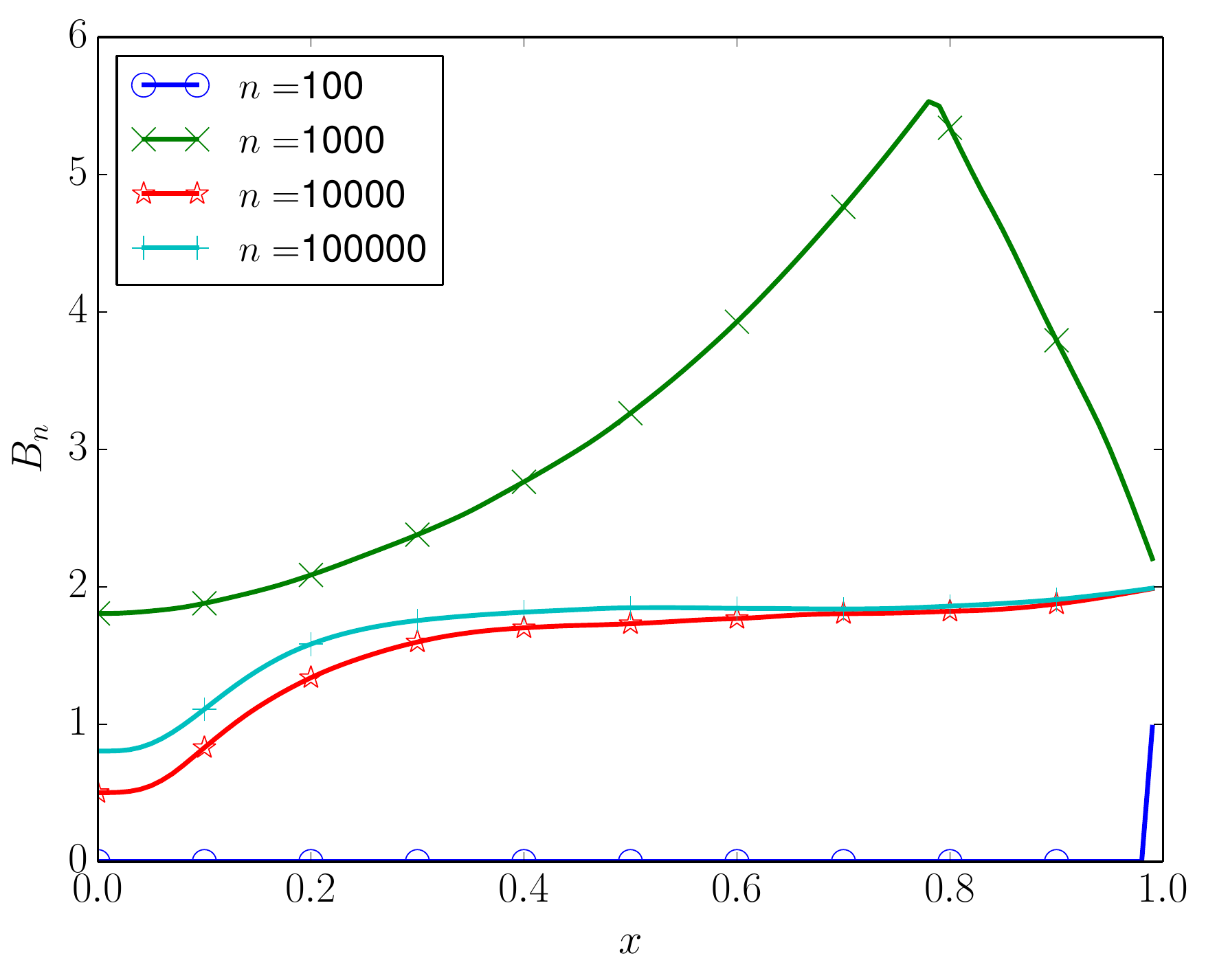}}

  \end{center}
  \caption{Convergence of the Robbins-Monro Algorithm \ref{a:RMforKL}
    applied to the Conditioned Diffusion problem, in the case of
    variable inverse covariance potential $B(t)$.  Figure (a) shows
    evolution of $m_n(t)$ with $n$; Figure (b) shows $m_n(t)$ at
    particular $n$.
    Figure (c) shows evolution of $B_n(t)$ with $n$; Figure (d) shows
    $B_n(t)$ at particular $n$.}
  \label{f:varB}

\end{figure}

\begin{figure}
  \begin{center}
    \subfigure{\includegraphics[width=6.25cm]{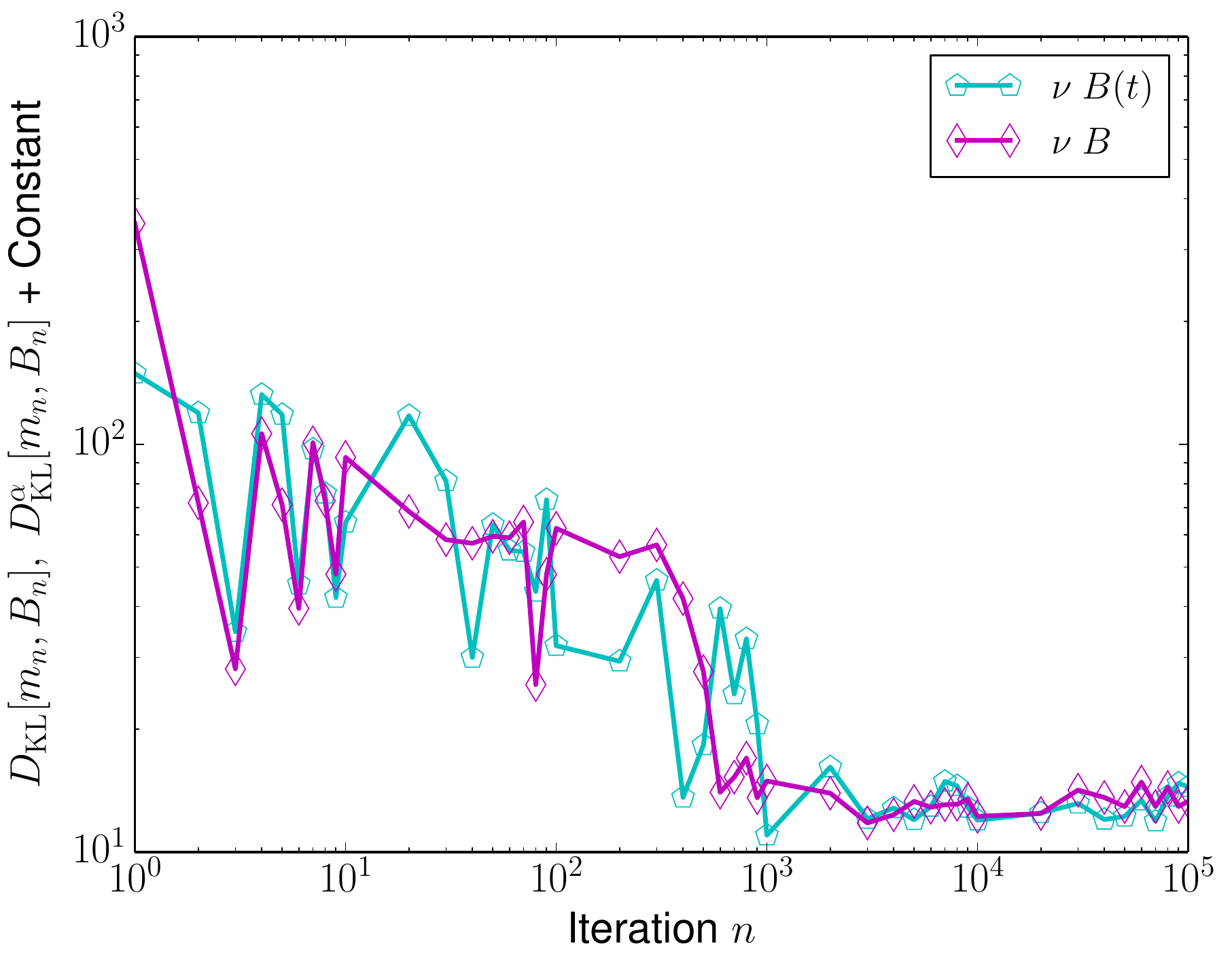}}
    \subfigure{\includegraphics[width=6.25cm]{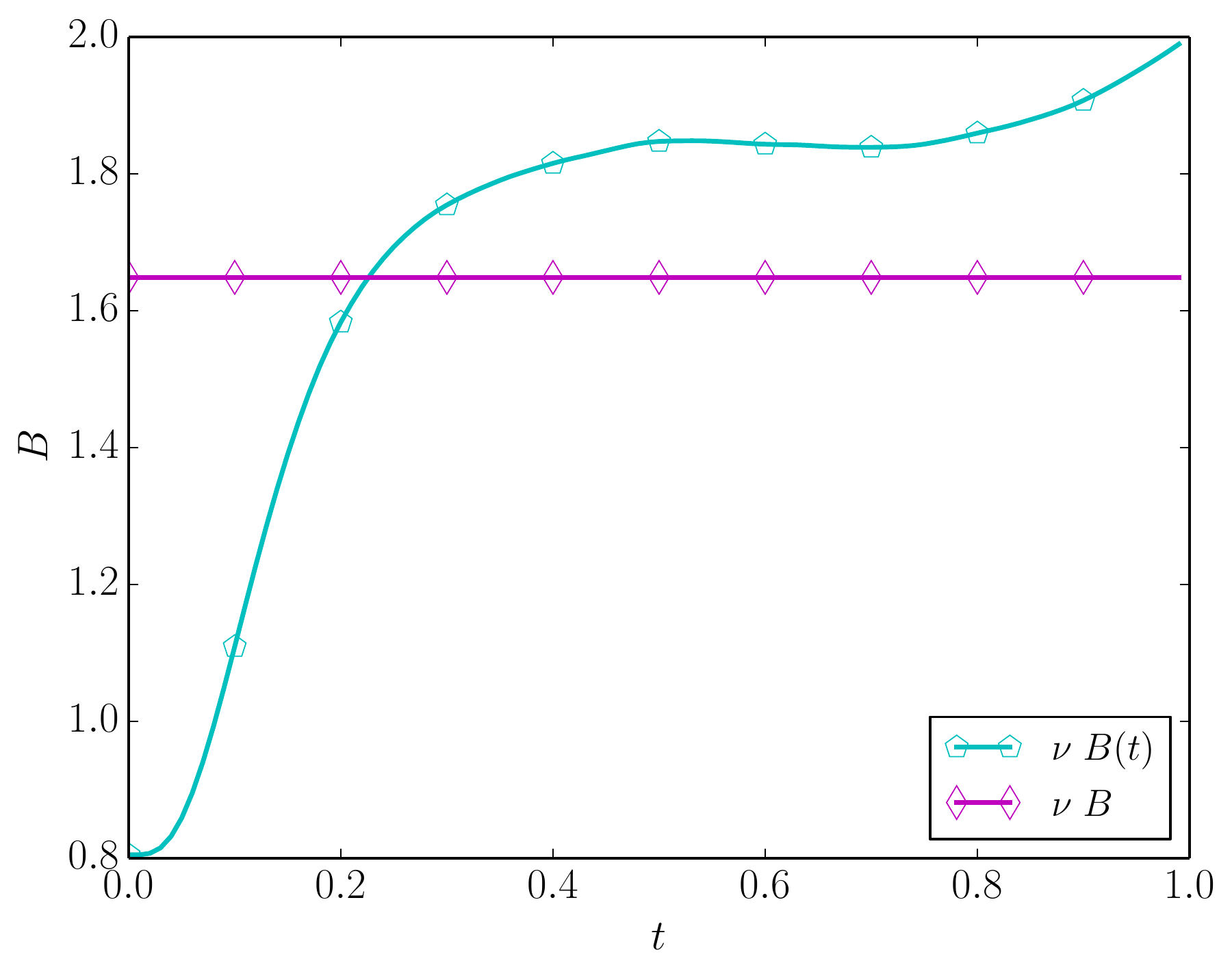}}
  \end{center}
  \caption{Minimization of $\Dkl^\alpha$ (for $B(t)$) and $\Dkl$ (for
    $B$) during Robbins-Monro Algorithm \ref{a:RMforKL} for the
    Conditioned Diffusion problem.  Also plotted is a comparison of
    $B$ and $B(t)$ for the optimized $\nu$ distributions.}
  \label{f:ac_dkl}
\end{figure}

The important comparison is when we sample the posterior using
these as the proposal distributions in MCMC Algorithms \ref{a1} and
\ref{a2}.  The results for this are given in Figure
\ref{f:allencahnpCN}.  Here, we compare both the prior and posterior
means and variances, along with the acceptance rates.  The means are
all in reasonable agreement, with the exception of the $m_0$, which
was to be expected.  The variances indicate that the sampling done
using $\mu_0$ has not quite converged, which is why it is far from the
posterior variances obtained from the optimized $\nu$'s, which are
quite close.  The optimized prior variances recover the plateau
between $t=0.2$ to $t=0.9$, but could not resolve the peak near $0.1$.
Variable $B(t)$ captures some of this information in that it has a
maximum in the right location, but of a smaller amplitude.  However,
when one standard deviation about the mean is plotted, it is difficult
to see this disagreement in variance between the reference and target
measures.

In Figure \ref{f:allencahnpCN_accept} we present the acceptance rate and
autocovariance, to assess the performance of Algorithms \ref{a1} and \ref{a2}.  For both the constant and variable potential cases,
there is better than an order of magnitude improvement over $\mu_0$.
In this case, it is difficult to distinguish an appreciable difference in
performance between $B(t)$ and $B$.

\begin{figure}
  \begin{center}
    \subfigure[Mean]{\includegraphics[width=6.25cm]{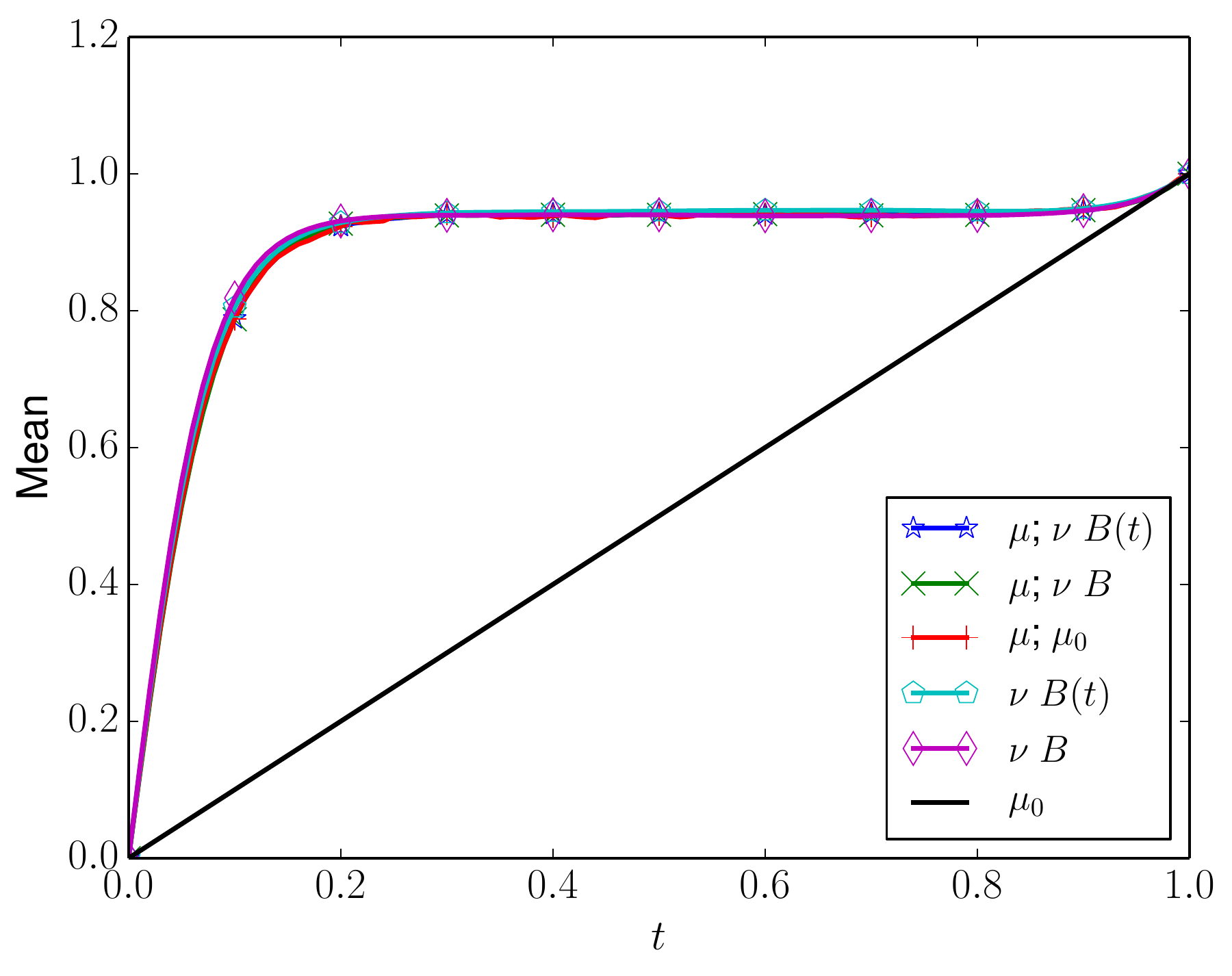}}
    \subfigure[Variance]{\includegraphics[width=6.25cm]{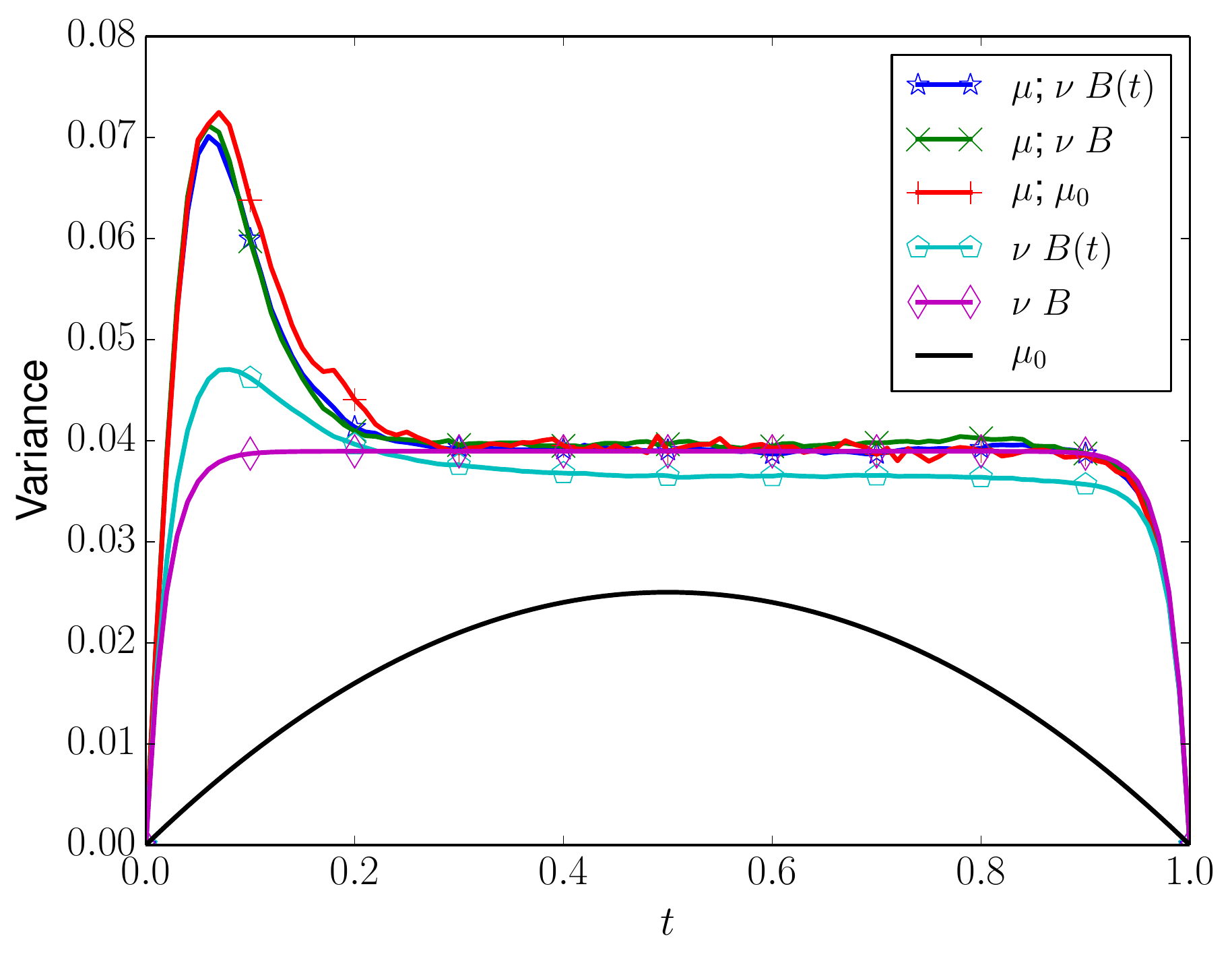}}

    \subfigure[$u(t$)]{\includegraphics[width=6.25cm]{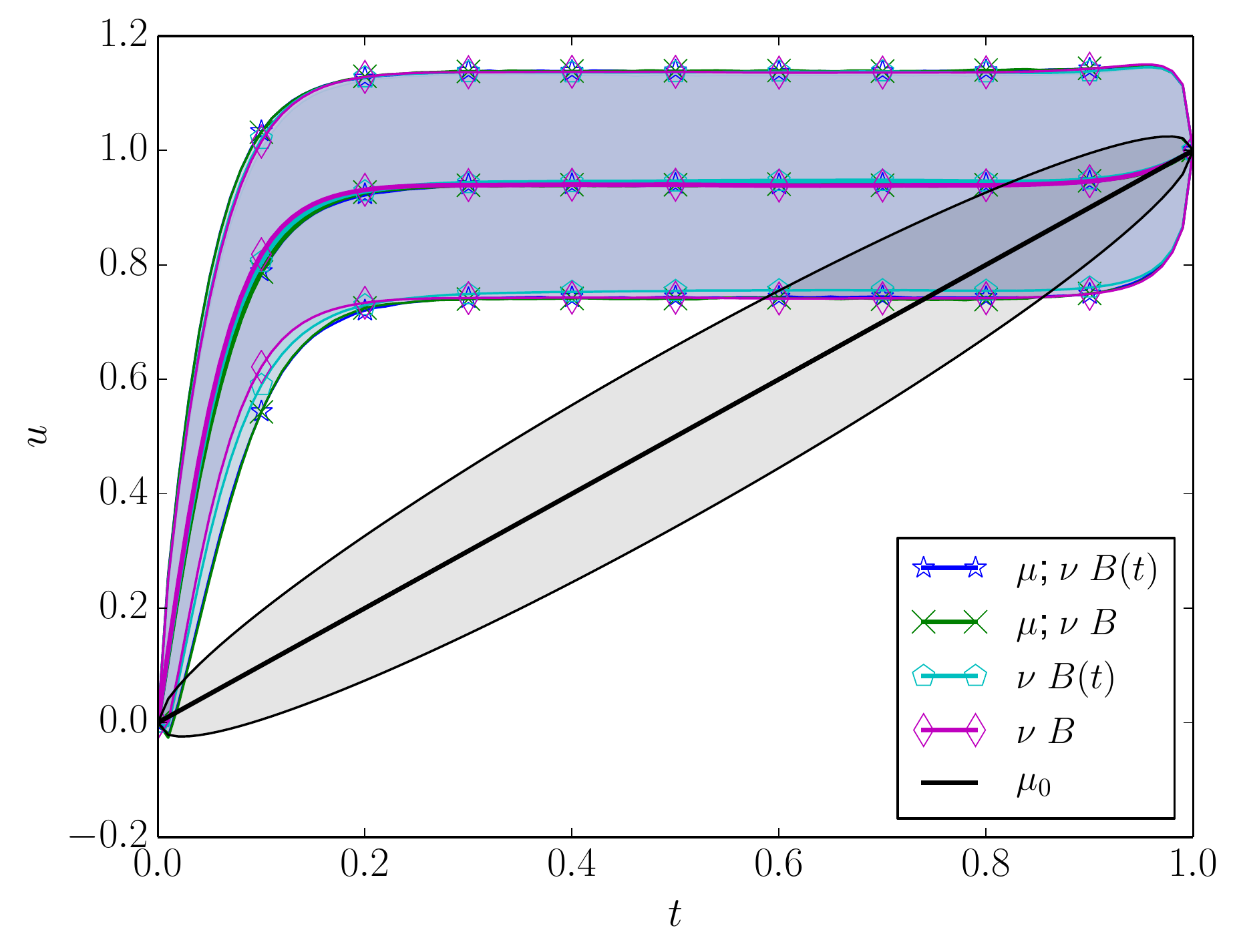}}
  \end{center}
  \caption{Behavior of MCMC Algorithms \ref{a1} and \ref{a2} for the
    Conditioned Diffusion problem.  The true posterior distribution,
    $\mu$, is sampled using $\mu_0$ (Algorithm \ref{a1}) and $\nu$,
    for both constant and variable potentials, $B$ and $B(t)$,
    (Algorithm \ref{a2}).  The resulting posterior approximations are
    denoted by $\mu; \mu_0$ (Algorithm \ref{a1}), and $\mu; \nu B$ and
    $\mu; \nu B(t)$ (Algorithm \ref{a2}).  The curves denoted $\mu_0$,
    and $\nu\; B$ and $\nu\; B(t)$, are the prior and best fit Gaussians.  For both optimized $\nu$'s, there is good
    agreement between the means and the posterior mean.  The variances
    are consistent, but the posterior shows a peak near $t=0.1$ that
    is not captured by $\nu$ distributions.  With the exception
    of $\mu_0$, there is good general agreement amongst the distributions of
    $u(t)$.  Shaded regions enclose $\pm$ one standard deviation.}
  \label{f:allencahnpCN}
\end{figure}

\begin{figure}

\begin{center}
  \subfigure[Acceptance
  Rate]{\includegraphics[width=6.25cm]{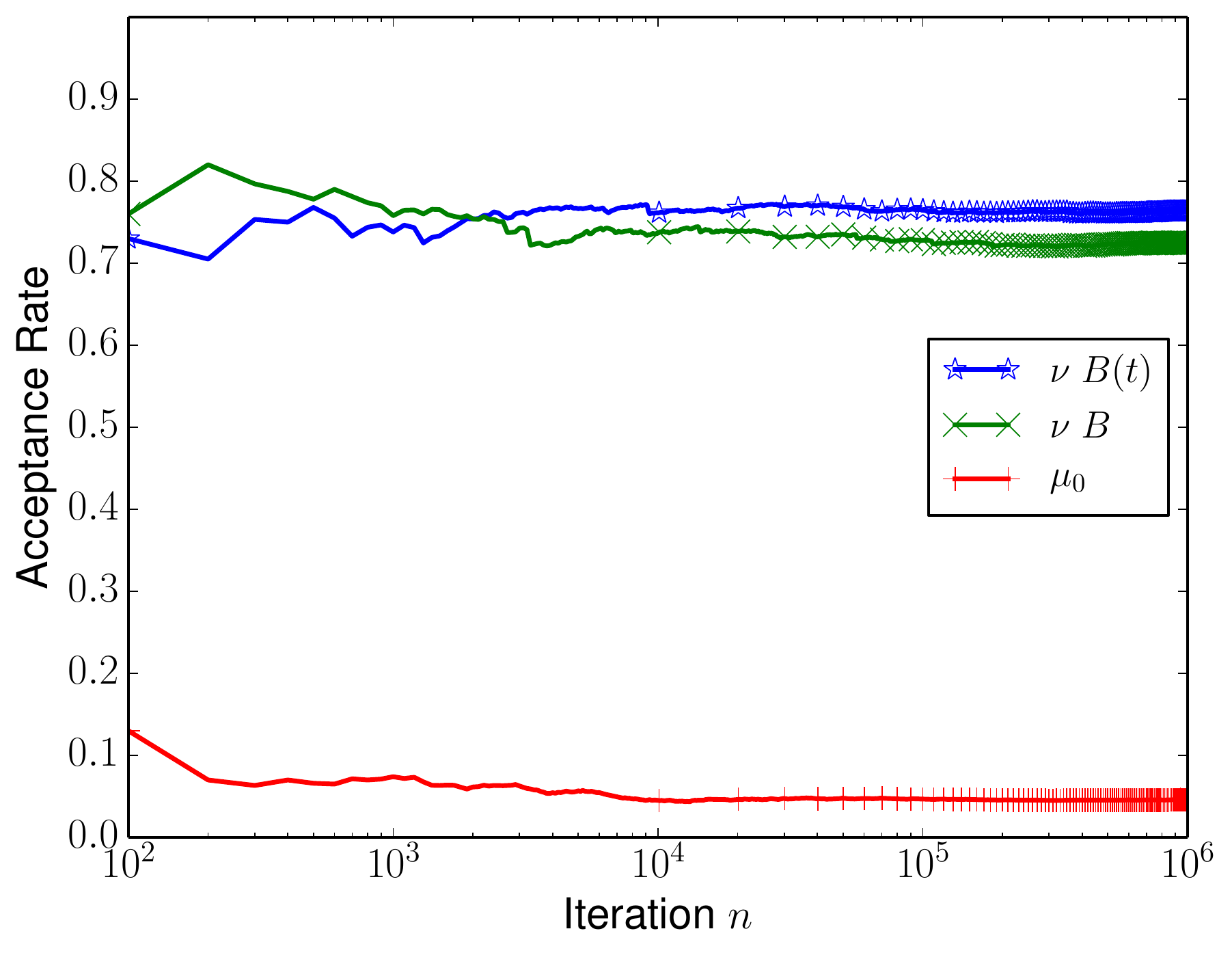}}
  \subfigure[Autocovariance]{\includegraphics[width=6.25cm]{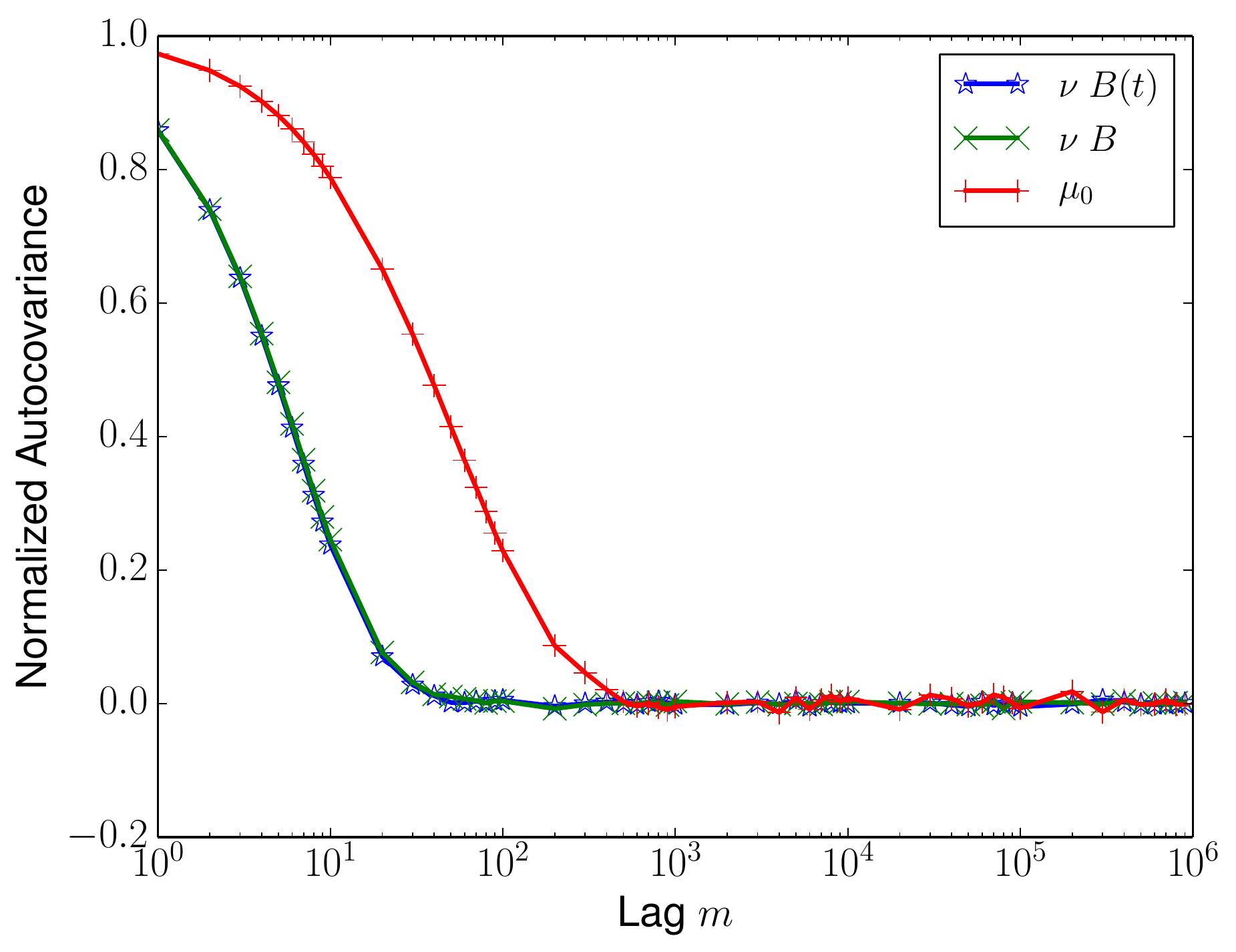}}
\end{center}
\caption{Performance of MCMC Algorithms \ref{a1} and \ref{a2} for the
  Conditioned Diffusion problem.  When $\mu_0$ is used for proposals
  in Algorithm \ref{a1}, the acceptance rate is far beneath either
  best fit Gaussian, $\nu\; B$ and $\nu\; B(t)$, within Algorithm
  \ref{a2}.  Variable $B(t)$ provides nominal improvement over
  constant $B$.}

\label{f:allencahnpCN_accept}
\end{figure}

\section{Conclusions}
\label{sec:C}
We have demonstrated a viable computational methodology for finding
the best Gaussian approximation to measures defined on a Hilbert space
of functions, using the Kullback-Leibler divergence as measure of fit.
We have parameterized the covariance via low rank matrices, or via a
Schr\"odinger potential in an inverse covariance representation, and
represented the mean nonparametrically, as a function; these
representations are guided by knowledge and understanding of the
properties of the underlying calculus of variations problem as
described in \cite{PSSW13}.  Computational results demonstrate that,
in certain natural parameter regimes, the Gaussian approximations are
good in the sense that they give estimates of mean and covariance
which are close to the true mean and covariance under the target
measure of interest, and that they consequently can be used to
construct efficient MCMC methods to probe the posterior distribution.

Further work is needed to explore the methodology in larger scale
applications and to develop application-specific parameterizations of
the covariance in this context. It would also be interesting to
combine the Robbins-Monro minimization with the MCMC method to
construct an adaptive MCMC method. On the analysis side it would be
instructive to demonstrate improved spectral gaps for the resulting
MCMC methods, with respect to observational noise (resp. temperature)
within the context of Bayesian inverse problems (resp. conditioned
diffusions), generalizing the analysis of Section
\ref{s:scalar_example}.

\appendix

\section{Scalar Example}
\label{s:SE}
In this section of the appendix we provide further details relating to
the motivational scalar example from section \ref{s:scalar_example}.

\subsection{Scalar Sampling}
\label{s:scalar}
Recall the scalar problem from Section \ref{s:scalar_example}.  One of
the motivations for considering such a problem is that many of the
calculations for $\Dnm$ are explicit.  Indeed, If $\nu = N(m,
\sigma^2)$ is the Gaussian which we intend to fit against $\mu$, then
\begin{equation}
  \label{e:scalar_dkl}
  \begin{split}
    \Dkl(\nu||\mu) &= \E^{\nu}\left[V(x) - \tfrac{1}{2\sigma^2}
      |x-m|^2\right] +\log Z_\mu - \log Z_\nu\\
    & = \E^{\nu_0}[V(y +m)] - \tfrac{1}{2} + \log Z_\mu - \log \sigma
    -
    \log\sqrt{2\pi},\\
    & = \E^{\nu_0}[V(y +m)]- \log \sigma + \text{Constant}.
  \end{split}
\end{equation}
The derivatives then take the simplified form
\begin{subequations}
  \label{e:dkl_derivs_scalar}
  \begin{align}
    D_m \Dkl & = \E^{\nu_0}[D_yV(y +m)], \\
    D_\sigma \Dkl & = \E^{\nu_0}[V(y +m)\sigma^{-3}(y^2- \sigma^2 )] -
    \sigma^{-1}.
  \end{align}
\end{subequations}
For some choices of $V(x)$, including \eqref{e:scalar_potential}, the
above expectations can be computed analytically, and the critical
points of \eqref{e:dkl_derivs_scalar} can then be obtained by
classical root finding.  Thus, we will be able to compare the
Robbins-Monro solution against a deterministic one, making for an
excellent benchmark problem.

The parameters used in these computation are:
\begin{itemize}
\item $10^6$ iterations of the Robbins-Monro with $10^2$ samples per
  iterations;
\item $a_0 = .1$ and $a_n = a_0 n^{-3/5}$;
\item $m_0 = 0$ and $\sigma_0 = 1$;
\item $m$ is constrained to the interval $[-.5,.5]$;
\item $\sigma$ is constrained to the interval $[10^{-3}, 10^0]$;
\item $10^6$ iterations of pCN, Algorithms \ref{a1}, \ref{a2}, are
  performed with $\beta = 1.$
\end{itemize}
While $10^6$ iterations of Robbins-Monro are used, Figure
\ref{f:scalar_conv} indicates that there is good agreement after
$10^3$ iterations. More iterations than needed are used in all of our
examples, to ensure convergence.  With appropriate convergence
diagnostics, it may be possible to identify a convenient termination
time.

\subsection{Analysis of the Sampling Performance}
\label{s:performance}
While the numerical experiments confirm our intuition, for this
example, the acceptance rate can be studied analytically.  Let
\begin{equation}
  \label{e:Tscalar}
  T(u,v) = \tfrac{1}{\eps}(u^4-v^4) +
  \left(\tfrac{1}{2\eps}-\tfrac{1}{2\sigma^2} \right)(u^2-v^2) + \tfrac{m}{\sigma^2}(u-v).
\end{equation} 
The acceptance probability for proposal $v$, given current state $u$,
is then $1 \wedge e^T$.  This is valid not only for our new Algorithm
\ref{a2}, using the optimized distribution $\nu=N(m,\sigma^2)$, but
also for Algorithm \ref{a1}, which uses the prior $\mu_0$, by taking
$m\mapsto 0$ and $\sigma \mapsto 1$ in \eqref{e:Tscalar}.

For an independence sampler, where proposals are generated solely from
$\nu$, we show that the expected acceptance rate of the optimized
$\nu$ tends to one as $\eps \to 0$. In contrast, when the prior,
$\mu_0 = N(0,1)$ is used as the proposal distribution, the acceptance
rate will be driven to zero.  We emphasize this case as the independence
sampler should have the poorest acceptance rate.  If instead of using
an independence sampler, we use a Crank-Nicolson proposal with
sufficiently small steps, favorable acceptance rates can be recovered
when $\mu_0$ is used for proposals.

These results are partially based on the following lemma, which
provides a lower bound on the acceptance rate:
\begin{lemma}[Lemma B.1 of \cite{Beskos:2009vx}]
  \label{l:accept}
  Let $Y$ be a real-valued random variable and $\gamma >0$.  Then
  \begin{equation*}
    \E[1\wedge e^Y] \geq e^{-\gamma} \left( 1 - \gamma^{-1}{\E[|Y|]} \right).
  \end{equation*}
\end{lemma}

\begin{proposition}
  \label{p:optimizednu}
  Assume $\nu$ is the $\Dkl$ optimized distribution for
  \eqref{e:scalar_dist} with potential \eqref{e:scalar_potential}.
  Furthermore, assume that $\mu^\eps$ is sampled using Algorithm
  \ref{a2} with $\beta=1$, and that it has reached stationarity.  Then
  $\E[|T|]\leq 18 \eps + \bigo(\eps^2)$, and for any fixed $\gamma>0$,
  $\lim_{\eps \to 0}\E[1\wedge e^T] \geq e^{-\gamma}$.
\end{proposition}

\begin{proof}
  First we estimate $\E[|T|]$, then we apply Lemma \ref{l:accept}.
  Since we are considering the case of the independence sampler, and
  have reached stationarity, we may take $u \sim \mu^\eps$ and $v \sim
  \nu$ to be independent.  Then, taking $m=0$ and $\sigma^2$ given by
  \eqref{e:sig_scalar},
  \begin{equation*}
    \begin{split}
      \E[|T|]&\leq\E^{\mu^\eps}\left[\tfrac{1}{\eps}u^4 + (6 +
        \bigo(\eps))u^2 \right] + \E^{\nu}\left[\tfrac{1}{\eps}v^4 +
        (6 + \bigo(\eps))v^2
      \right]\\
      &\leq 3\eps + 6 \eps + 3\eps + 6 \eps + \bigo(\eps^2).
    \end{split}
  \end{equation*}
  Details of the moment estimates are given in Section
  \ref{e:accept_details}.  The result is now obvious.
\end{proof}

\begin{proposition}
  \label{p:nonoptimizednu}
  Assume that $\mu^\eps$ is sampled using Algorithm \ref{a1} with
  $\beta=1$, and that it has reached stationarity.  Then $\E[1\wedge
  e^T] \lesssim\eps^{1/2}$.
\end{proposition}

\begin{proof}
  The strategy is to make estimates using a Gaussian in place of
  $\mu^\eps$.  Let $\tilde\mu^\eps = N(0, \eps)$, and when denote
  $\tilde u \sim \tilde\mu^\eps$ to distinguish it form $\mu^\eps$.
  Then, since $1 \wedge e^T\geq 0$,
  \begin{equation*}
    \E[1\wedge e^T] = \E[1\wedge e^{T(u,v)}] \leq
    \frac{\sqrt{2\pi \eps}}{Z_\eps}\E[1\wedge
    e^{T(\tilde u,v)}] = (1+\bigo(\eps))\E[1\wedge
    e^{T(\tilde u,v)}].
  \end{equation*}
  The estimate of $Z_\eps$ is given in Section \ref{e:accept_details},
  and
  \begin{equation}
    \label{e:accept1}
    \E[1\wedge
    e^{T(\tilde u,v)}]  = \E[
    e^{T(\tilde u,v)} 1_{{T(\tilde u,v)}<0}]  + \bbP[{{T(\tilde u,v)}\geq 0}].
  \end{equation}
  Observe now that \eqref{e:Tscalar} can be factored, and for $m=0$,
  $\sigma =1$, which is the case here,
  \begin{equation*}
    {T(\tilde u,v)} = (\tilde u^2-v^2) \left(\tfrac{1}{\eps}(\tilde u^2+v^2) + \tfrac{1-\eps}{2\eps}\right).
  \end{equation*}
  For $\eps <1$, $T\gtrless 0$ if and only if $\tilde u^2 \gtrless
  v^2$.  Using explicit integration, detailed in Section
  \ref{e:accept_details}, $\bbP[{T(\tilde u,v)\geq 0}] =
  \tfrac{2}{\pi}\arctan(\eps^{1/2})$.  For the other term in
  \eqref{e:accept1}, since the expectation is over the region $u^2 <
  v^2$, $ {T(\tilde u,v)}\leq \tfrac{1-\eps}{2\eps} (\tilde u^2-v^2),$
  so that
  \begin{equation*}
    \E[
    e^{T(\tilde u,v)}1_{{T(\tilde u,v)}<0}]\leq \E[
    \exp\left\{ \tfrac{1-\eps}{2\eps} (\tilde u^2-v^2) \right
    \}1_{\tilde u^2< v^2}]=  \tfrac{2}{\pi}\arctan(\eps^{1/2}).
  \end{equation*}
\end{proof}

\begin{proposition}
  \label{p:nonoptimizednubeta}
  Assume that $\mu^\eps$ is sampled using Algorithm \ref{a1} with
  $\beta = \eps<1$, and that it has reached stationarity.  Then
  $\E[|T|]\lesssim \eps^{1/2}$, and for any fixed $\gamma>0$,
  $\lim_{\eps \to 0}\E[1\wedge e^T] \geq e^{-\gamma}$.
\end{proposition}

\begin{proof}
  Now the proposal $v$ depends on $u$, according to
  \begin{gather*}
    v = \sqrt{1-\eps^2} u + \eps \xi, \quad \xi \sim \mu_0\\
    v - u = (\sqrt{1-\eps^2}-1) u + \eps \xi.
  \end{gather*}
  Notice that for small $\eps$, $\sqrt{1-\eps^2}-1 =
  -\tfrac{1}{2}\eps^2 + \big(\eps^4)$.  The idea is to use continuity
  of the functional, since $v$ is close to $u$, to get an upper bound
  on $\E|T|$, and then apply Lemma \ref{l:accept}.  Using conditioning
  and estimates of the moments found in Section
  \ref{e:accept_details},
  \begin{equation*}
    \begin{split}
      \E|T|&\leq \eps^{-1} \E[|u^4 - v^4|] +
      \tfrac{1-\eps}{2\eps}\E[|u^2-v^2|] \\
      &\lesssim \eps^{-1} \eps^{9/4} + \eps^{-1} \eps^{3/2}=
      \eps^{5/4} + \eps^{1/2}.
    \end{split}
  \end{equation*}
\end{proof}

Note that Algorithm \ref{a1} is equivalent to Algorithm \ref{a2}, when
$\nu = \mu_0$. However the preceding three propositions show the
advantages that result from use of Algorithm \ref{a2} when using a
well-chosen $\nu$. In particular the independence sampler ($\beta=1$)
accepts at rate which is $\epsilon$ independent, resulting in rapid
decorrelation of the Markov chain. In contrast, Algorithm \ref{a1}
with $\beta=1$ has acceptance probability which degenerates as
$\epsilon \to 0$, inducing slow decorrelation in the Markov chain; an
${\cal O}$(1) acceptance probability can be achieved for Algorithm
\ref{a1}, but this requires choosing $\beta={\cal O}(1)$, also
inducing slow decorrelation. In summary the results demonstrate
analytically the advantages of using Algorithm \ref{a2}.

\subsection{Details of the Acceptance Rate Estimates}
\label{e:accept_details}

\subsubsection{Moment Estimates}
Moments of $\mu^\eps$ are needed, which can be estimated using the
bound
\begin{equation}
  \left(1 - \frac{1}{\eps} x^4\right)  \exp\left(-\frac{x^2}{2\eps}\right)\leq \exp\left(-\frac{1}{\eps}
    V(x)\right)\leq \exp\left(-\frac{x^2}{2\eps}\right).
\end{equation}
We can then estimate the partition function and the moments:
\begin{subequations}
  \begin{align}
    Z_\eps &= \sqrt{2\pi \eps}(1  + \bigo(\eps)),\\
    \E^{\mu^\eps}[u^2] &= \eps + \bigo(\eps^2),\\
    \E^{\mu^\eps}[u^4] &= 3\eps^2 + \bigo(\eps^3),\\
    \E^{\mu^\eps}[u^6] &= 15\eps^3 + \bigo(\eps^4).
  \end{align}
\end{subequations}

\subsubsection{Upper Bound Estimates}
In the proof of Proposition \ref{p:nonoptimizednu}, the two terms in
\eqref{e:accept1} can be integrated explicitly.  This is done by
letting $V = v^2$ and $W = \tilde u^2/\eps$, so that $V$ and $W$ are
independent $\chi^2$ variables.  Then $T\geq 0$ corresponds to $W\geq
V/\eps$, and
\begin{equation*}
  \begin{split}
    \bbP[{T(\tilde u, v)\geq 0}] &= \int_{V=0}^{\infty} \left\{\int_{W
        = V/\eps}^{\infty}\chi^2(dW) \right\}\chi^2(dV)\\
    & = \int_{V=0}^{\infty}
    \left\{\Erfc\left(\sqrt{\tfrac{V}{2\eps}}\right)\right\}\chi^2(dV)=\tfrac{2}{\pi}\arctan(\eps^{1/2}).
  \end{split}
\end{equation*}
Analogously,
\begin{equation*}
  \begin{split}
    &\E\left[\exp\left\{ \tfrac{1-\eps}{2\eps}(\eps W -
        V)\right\}1_{T(\tilde u,v)<0}\right]\\
    &=\int_{V=0}^{\infty} \left \{\int_{W = 0}^{V/\eps} \exp\left\{
        \tfrac{1-\eps}{2\eps}(\eps W -
        V)\right\}\chi^2(dW)\right \}\chi^2(dV)  \\
    & =\int_{V=0}^{\infty}\left\{\exp
      \left\{-\tfrac{1-\eps}{2\eps}V\right\}\eps^{-1/2}\Erf\right(\sqrt{\tfrac{V}{2}}
    \left) \right\}\chi^2(dV) =\frac{2}{\pi}\arctan(\eps^{1/2}).
  \end{split}
\end{equation*}

\subsubsection{Estimates for Crank-Nicolson Proposals}
The last quantities we need are the differences appearing in the proof
of Proposition \ref{p:nonoptimizednubeta}:
\begin{subequations}
  \begin{align}
    \E[|u^2-v^2|]&\leq \sqrt{\E[|u+v|^2]}\sqrt{\E[|u-v|^2]},\\
    \E[|u^4-v^4|]&\leq \sqrt{\E[|u^3+u^2v + uv^2
      +v^3|^2]}\sqrt{\E[|u^2-v^2|]}.
  \end{align}
\end{subequations}
Using the definition of the proposal $v$ and the estimates of the
moments of $\mu^\eps$,
\begin{equation*}
  \begin{split}
    \sqrt{\E[|u-v|^2]} &= \sqrt{\E[|(\sqrt{1-\eps^2}-1) u + \eps
      \xi|^2]}\\
    &\leq (\tfrac{1}{2}\eps^2 + \bigo(\eps^4) )
    \sqrt{\E^{\mu^{\eps}}[u^2]} + \eps \sqrt{\E^{\mu_0}[\xi^2]}\\
    &\leq \tfrac{1}{2}\eps^{5/2} + \eps + \bigo(\eps^{7/2}),
  \end{split}
\end{equation*}
and
\begin{equation*}
  \begin{split}
    \sqrt{\E[|u+v|^2]} &= \sqrt{\E[|(\sqrt{1-\eps^2}+1) u + \eps
      \xi|^2]}\\
    &\leq (2 + \bigo(\eps^2) )
    \sqrt{\E^{\mu^{\eps}}[u^2]} + \eps \sqrt{\E^{\mu_0}[\xi^2]}\\
    &\leq 2\eps^{1/2} + \eps + \bigo(\eps^{3/2}).
  \end{split}
\end{equation*}
Consequently, $\E[|u^2-v^2|] \leq 2 \eps^{3/2} + \bigo(\eps^2)$.  The
cubic term can be bounded as
\begin{equation*}
  \begin{split}
    \sqrt{\E[|u^3+u^2v + uv^2 +v^3|^2]}& \leq  \sqrt{\E[u^6]} +\sqrt{\E[u^4v^2]}+ \sqrt{\E[u^2v^4]}+ \sqrt{\E[v^4]}\\
    & \leq \E[u^6]^{1/2} + \E[u^6]^{1/3}\E[v^6]^{1/6}\\
    &\quad + \E[u^6]^{1/6}\E[v^6]^{1/3}+\E[v^6]^{1/2}.
  \end{split}
\end{equation*}
Thus, the final estimate is
\begin{equation*}
  \begin{split}
    \E[v^6]^{1/6} =\E[|(\sqrt{1-\eps^2}) u + \eps \xi|^6]^{1/6} &\leq
    (1+\bigo(\eps))\E^{\mu^\eps}[u^6]^{1/6} + \eps
    \E^{\mu_0}[v^6]^{1/6}\\
    &\leq (1+\bigo(\eps))((15)^{1/6} \eps^{1/2} + \bigo(\eps^{3/2})) +
    \eps\\
    &\quad = 15^{1/6} \eps^{1/2} + \eps + \bigo(\eps^{3/2}).
  \end{split}
\end{equation*}
Therefore, $\E[|u^4-v^4|]\lesssim \eps^{3/2} \cdot \eps^{3/4} =
\eps^{9/4}$.

\section{Sample Generation}
\label{a:samples}
In this section of the appendix we briefly comment on how samples were
generated to estimate expectations and perform pCN sampling of the
posterior distributions.  Three different methods were used

\subsection{Bayesian Inverse Problem}
For the Bayesian inverse problem presented in Section \ref{ssec:BIP},
samples were drawn from $N(0, C)$, where $C$ was a finite rank
perturbation of $C_0$, $C_0^{-1} = \delta^{-1} (-d^2/dx^2)$ equipped
with periodic boundary conditions on $[0,1)$.  This was accomplished
using the Karhunen Lo{\`e}ve series expansion (KLSE) and the fast
Fourier transform (FFT).  Observe that the spectrum of $C_0$ is:
\begin{equation}
  \varphi_n(x) = \begin{cases}
    \sqrt{2}\sin(2\pi \tfrac{n+1}{2} x) & \text{$n$ odd},\\
    \sqrt{2}\cos(2\pi \tfrac{n}{2} x) & \text{$n$ even},
  \end{cases}, \quad \lambda_n^2 = \begin{cases}
    \frac{\delta}{(2\pi \tfrac{n+1}{2})^2}& \text{$n$ odd},\\
    \frac{\delta}{(2\pi \tfrac{n}{2})^2}& \text{$n$ even}.
  \end{cases}
\end{equation}
Let ${\bf x}^n$ and $\mu_n^2$ denote the normalized eigenvectors and
eigenvalues of matrix $B$ of rank $K$. Then if $u \sim N(0,C)$, $\xi_n
\sim N(0,1)$, i.i.d., the KLSE is:
\begin{equation}
  u = \sum_{\ell =1}^K \left\{\sum_{n=1}^K{\mu_n} \xi_n
    x_\ell^n\right\}\varphi_{\ell}(x)+\sum_{\ell = K+1}^\infty{\lambda_\ell} \xi_\ell\varphi_\ell(x)
\end{equation}
Truncating this at some index, $N>K$, we are left with a trigonometric
polynomial which can be evaluated by FFT.  This will readily adapt to
problems posed on the $d$-dimensional torus.

\subsection{Conditioned Diffusion with Constant Potential}
For the conditioned diffusion in Section \ref{ssec:CDP}, the case of
the constant potential $B$ can easily be treated, as this corresponds
to an Ornstein-Uhlenbeck (OU) bridge.  Provided $B>0$ is constant, we
can associate to $N(0, C)$ the conditioned OU bridge:
\begin{equation}
  dy_t = \eps^{-1} \sqrt{B} y_t dt + \sqrt{2} dw_t, \quad y_0= y_1 = 0,
\end{equation}
and the unconditioned OU process
\begin{equation}
  dz_t = \eps^{-1} \sqrt{B} z_t dt + \sqrt{2} dw_t, \quad z_0= 0.
\end{equation}
Using the relation
\begin{equation}
  y_t = z_t - \frac{\sinh(\sqrt{B}t/\eps)}{\sinh(\sqrt{B}/\eps)}z_1,
\end{equation}
if we can generate a sample of $z_t$, we can then sample from
$N(0,C)$.  This is accomplished by picking a time step $\Delta t>0$,
and then iterating:
\begin{equation}
  z_{n+1} = \exp\{-\eps^{-1} \sqrt{B} \Delta t\}z_n + \eta_n, \quad
  \eta_n \sim N(0,\eps/\sqrt{B} (1-\exp(-2\eps^{-1} \sqrt{B} \Delta t)).
\end{equation}
Here, $z_0 = 0$, and $z_n\approx z_{n \Delta t}$.  This is highly
efficient and generalizes to $d$-dimensional diffusions.

\subsection{Conditioned Diffusion with Variable Potential}
Finally, for the conditioned diffusion with a variable potential
$B(t)$, we observe that for the Robbins-Monro algorithm, we do not
need the samples themselves, but merely estimates of the expectations.
Thus, we employ a change of measure so as to sample from a constant
$B$ problem, which is highly efficient.  Indeed, for any observable
$\mathcal{O}$,
\begin{equation}
  \E^{\nu_0}[\mathcal{O}] = \E^{\bar\nu} [\mathcal{O}\tfrac{d\nu_0}{d\bar\nu}]=\frac{\E^{\bar\nu} [\mathcal{O}\exp(-\Psi)]}{\E^{\bar\nu} [\exp(-\Psi)]}
\end{equation}
Formally,
\begin{equation}
  \frac{d\nu_0}{d\bar\nu} \propto \exp\left\{- \frac{1}{4\eps^2}\int_0^1
    (B(t) - \bar B) z_t^2 dt\right\},
\end{equation}
and we take $\bar B = \max_t B(t)$ for stability.

For pCN sampling we need actual samples from $N(0,C)$.  We again use a\\
Karhunen-Lo{\`e}ve series expansion, after discretizing the precision
operator $C^{-1} = C_0^{-1} +B(t)$ with appropriate boundary
conditions, and computing its eigenvalues and eigenvectors.  While
this computation is expensive, it is only done once at the beginning
of the posterior sampling algorithm.

\vspace{0.1in}
\noindent{\bf Acknowledgements} AMS is grateful to EPSRC, ERC and ONR
for financial support. He is also grateful to Folkmar Bornemann for
helpful discussions concerning paramaterization of the covariance
operator.

\noindent FJP would like to acknowledge the hospitality of the
University of Warwick during his stay. 

\noindent GS was supported in part by DOE Award DE-SC0002085 and NSF
PIRE Grant OISE-0967140.

\noindent HW was supported by an EPSRC First Grant.

\vspace{0.1in}

\bibliographystyle{siam} \bibliography{KL}
\end{document}